\documentclass{tac}

\usepackage{bussproofs}
\usepackage{amsmath,amssymb,amsfonts,graphicx}
\usepackage{xypic}
\usepackage{stmaryrd}
\usepackage{tikz}
\usepackage{proof}
\usepackage[shortlabels]{enumitem}
\usepackage{mathtools}
\usepackage{lscape}
\usepackage{tocloft}
\usepackage{multicol} 
\usepackage{leftidx}
\usepackage{bbm}
\usepackage{rotating}
\usepackage{extarrows}
\usepackage{placeins}
\usepackage{caption}
\usepackage{tikz-cd}
\usepackage{dutchcal}
\usepackage[capitalize,poorman]{cleveref}
\usepackage{floatpag}

\usepackage[colorlinks=true]{hyperref}
\hypersetup{allcolors=[rgb]{0.1,0.1,0.4}}

\newcounter{dummy} 
\numberwithin{dummy}{section}

\newtheorem{lemma}[dummy]{Lemma}  
\newtheorem{theorem}[dummy]{Theorem}
\newtheorem{definition}[dummy]{Definition}

\newtheorem{proposition}[dummy]{Proposition} 
\newtheorem{corollary}[dummy]{Corollary} 
\newtheorem{example}[dummy]{Example}

\numberwithin{equation}{section}

\newcommand{\tri}{\triangleleft}

\newcommand{\X}{\mathbb{X}}
\newcommand{\Y}{\mathbb{Y}}
\newcommand{\C}{\mathbb{C}}
\newcommand{\m}{{\sf m}}

\newcommand{\Set}{{\sf Set}}
\newcommand{\duo}{\mathsf{duo}}
\newcommand{\indep}{{\sf indep}}
\newcommand{\Poly}{{\sf Poly}}
\newcommand{\lCore}{{\sf core_{\ell}}}
\newcommand{\rCore}{{\sf core_r}}
\newcommand{\nD}{{\sf normalDuo}}
\newcommand{\Iso}{{\sf Isomix}}
\newcommand{\coeval}{{\sf coev}}
\newcommand{\eval}{{\sf ev}}

\newcommand{\dashvv}{\dashv \!\!\! \dashv}  
\newcommand{\yon}{\mathcal{y}}
\newcommand{\then}{\fatsemi}

\newcommand{\id}{{\sf id}} 
\newcommand{\ox}{\otimes}
\newcommand{\oa}{\oplus}
\newcommand{\op}{\mathsf{op}}
\newcommand{\dual}{\mathbin{\text{\reflectbox{$\Vdash$}}}}
\newcommand{\lollipop}{\ensuremath{\!-\!\!\circ}}

\newlength{\llcfoo}



\makeatletter

\newdimen\w@dth

\def\setw@dth#1#2{\setbox\z@\hbox{\scriptsize $#1$}\w@dth=\wd\z@
\setbox\@ne\hbox{\scriptsize $#2$}\ifnum\w@dth<\wd\@ne \w@dth=\wd\@ne \fi
\advance\w@dth by 1.2em}

\def\t@^#1_#2{\allowbreak\def\n@one{#1}\def\n@two{#2}\mathrel
{\setw@dth{#1}{#2}
\mathop{\hbox to \w@dth{\rightarrowfill}}\limits
\ifx\n@one\empty\else ^{\box\z@}\fi
\ifx\n@two\empty\else _{\box\@ne}\fi}}
\def\t@@^#1{\@ifnextchar_ {\t@^{#1}}{\t@^{#1}_{}}}

\def\t@left^#1_#2{\def\n@one{#1}\def\n@two{#2}\mathrel{\setw@dth{#1}{#2}
\mathop{\hbox to \w@dth{\leftarrowfill}}\limits
\ifx\n@one\empty\else ^{\box\z@}\fi
\ifx\n@two\empty\else _{\box\@ne}\fi}}
\def\t@@left^#1{\@ifnextchar_ {\t@left^{#1}}{\t@left^{#1}_{}}}

\def\two@^#1_#2{\def\n@one{#1}\def\n@two{#2}\mathrel{\setw@dth{#1}{#2}
\mathop{\vcenter{\hbox to \w@dth{\rightarrowfill}\kern-1.7ex
                 \hbox to \w@dth{\rightarrowfill}}%
       }\limits
\ifx\n@one\empty\else ^{\box\z@}\fi
\ifx\n@two\empty\else _{\box\@ne}\fi}}
\def\tw@@^#1{\@ifnextchar_ {\two@^{#1}}{\two@^{#1}_{}}}

\def\tofr@^#1_#2{\def\n@one{#1}\def\n@two{#2}\mathrel{\setw@dth{#1}{#2}
\mathop{\vcenter{\hbox to \w@dth{\rightarrowfill}\kern-1.7ex
                 \hbox to \w@dth{\leftarrowfill}}%
       }\limits
\ifx\n@one\empty\else ^{\box\z@}\fi
\ifx\n@two\empty\else _{\box\@ne}\fi}}
\def\t@fr@^#1{\@ifnextchar_ {\tofr@^{#1}}{\tofr@^{#1}_{}}}

\newdimen\W@dth
\def\setW@dth#1#2{\setbox\z@\hbox{$#1$}\W@dth=\wd\z@
\setbox\@ne\hbox{$#2$}\ifnum\W@dth<\wd\@ne \W@dth=\wd\@ne \fi
\advance\W@dth by 1.2em}

\def\T@^#1_#2{\allowbreak\def\N@one{#1}\def\N@two{#2}\mathrel
{\setW@dth{#1}{#2}
\mathop{\hbox to \W@dth{\rightarrowfill}}\limits
\ifx\N@one\empty\else ^{\box\z@}\fi
\ifx\N@two\empty\else _{\box\@ne}\fi}}
\def\T@@^#1{\@ifnextchar_ {\T@^{#1}}{\T@^{#1}_{}}}

\def\T@left^#1_#2{\def\N@one{#1}\def\N@two{#2}\mathrel{\setW@dth{#1}{#2}
\mathop{\hbox to \W@dth{\leftarrowfill}}\limits
\ifx\N@one\empty\else ^{\box\z@}\fi
\ifx\N@two\empty\else _{\box\@ne}\fi}}
\def\T@@left^#1{\@ifnextchar_ {\T@left^{#1}}{\T@left^{#1}_{}}}

\def\Tofr@^#1_#2{\def\N@one{#1}\def\N@two{#2}\mathrel{\setW@dth{#1}{#2}
\mathop{\vcenter{\hbox to \W@dth{\rightarrowfill}\kern-1.7ex
                 \hbox to \W@dth{\leftarrowfill}}%
       }\limits
\ifx\N@one\empty\else ^{\box\z@}\fi
\ifx\N@two\empty\else _{\box\@ne}\fi}}
\def\T@fr@^#1{\@ifnextchar_ {\Tofr@^{#1}}{\Tofr@^{#1}_{}}}

\def\Two@^#1_#2{\def\N@one{#1}\def\N@two{#2}\mathrel{\setW@dth{#1}{#2}
\mathop{\vcenter{\hbox to \W@dth{\rightarrowfill}\kern-1.7ex
                 \hbox to \W@dth{\rightarrowfill}}%
       }\limits
\ifx\N@one\empty\else ^{\box\z@}\fi
\ifx\N@two\empty\else _{\box\@ne}\fi}}
\def\Tw@@^#1{\@ifnextchar_ {\Two@^{#1}}{\Two@^{#1}_{}}}

\def\to{\@ifnextchar^ {\t@@}{\t@@^{}}}
\def\from{\@ifnextchar^ {\t@@left}{\t@@left^{}}}
\def\tofro{\@ifnextchar^ {\t@fr@}{\t@fr@^{}}}
\def\To{\@ifnextchar^ {\T@@}{\T@@^{}}}
\def\From{\@ifnextchar^ {\T@@left}{\T@@left^{}}}
\def\Two{\@ifnextchar^ {\Tw@@}{\Tw@@^{}}}
\def\Tofro{\@ifnextchar^ {\T@fr@}{\T@fr@^{}}}

\makeatother

\tikzstyle{strings}=[baseline={([yshift=-.5ex]current bounding box.center)}]

\tikzset{every picture/.append style={scale=.5}, transform shape, strings}

\tikzset{%
symbol/.style={%
draw=none,
every to/.append style={%
edge node={node [sloped, allow upside down, auto=false]{$#1$}}}
}
}

\usetikzlibrary{shapes.geometric}
\usetikzlibrary{patterns}
\usetikzlibrary{fit}
\usetikzlibrary{positioning}
\usetikzlibrary{calc}
\usetikzlibrary{arrows}
\usetikzlibrary{decorations.markings}
\usetikzlibrary{decorations.pathreplacing}
\usetikzlibrary{shapes}

\pgfdeclarelayer{nodelayer}
\pgfdeclarelayer{edgelayer}
\pgfsetlayers{edgelayer,nodelayer,main}

\tikzset{simple/.style={}}
\tikzset{nothing/.style={outer sep=-3.4pt}}
\tikzset{map/.style={draw,fill=white, rectangle}}
\tikzstyle{filled}=[-, fill=black]

\tikzset{dot/.style={thick, fill=black, circle, scale=1, inner sep = .05cm}}

\tikzset{oa/.style={draw, scale=0.9,minimum height=.1cm,circle,append after command={
[shorten >=\pgflinewidth, shorten <=\pgflinewidth,]
(\tikzlastnode.north) edge (\tikzlastnode.south)
(\tikzlastnode.east) edge (\tikzlastnode.west)
} } }

\tikzset{ox/.style={draw, scale=0.9,minimum height=.1cm,circle,append after command={
[shorten >=\pgflinewidth, shorten <=\pgflinewidth,]
(\tikzlastnode.north west) edge (\tikzlastnode.south east)
(\tikzlastnode.north east) edge (\tikzlastnode.south west) } } }

\tikzset{circ/.style={
shape=circle, inner sep=1pt, draw}}

\tikzstyle{none}=[inner sep=-1pt]
\tikzstyle{circle}=[shape=circle,draw]

\tikzstyle{onehalfcircle}=[shape=circle, scale=1.5, draw]
\tikzstyle{twocircle}=[shape=circle, scale=2, draw]
\tikzstyle{black}=[shape=circle, fill=black, draw]

\tikzset{wires/.style={}}

\tikzset{box/.style={inner sep=0pt, thick, draw=black, text height=1.5ex, text depth=.25ex, 
text centered, minimum height=3em, anchor=center}}

\tikzcdset{every label/.append style = {font = \Large}}

\newcommand{\linmonw} {\xymatrixcolsep{4mm} \xymatrix{ \ar@{-||}[r]^{\circ} & }}
\newcommand{\linmonwl} {\xymatrixcolsep{4mm} \xymatrix{ \ar@{-||}[r]^{\otimes\;\tri} & }}
\newcommand{\linmonwr} {\xymatrixcolsep{4mm} \xymatrix{ \ar@{-||}[r]^{\tri\;\otimes} & }}

\newcommand{\lincomonb} {\xymatrixcolsep{4mm} \xymatrix{ \ar@{-||}[r]_{\bullet} & }}
\newcommand{\lincomonw} {\xymatrixcolsep{4mm} \xymatrix{ \ar@{-||}[r]_{\circ} & }}
\newcommand{\lincomonwr} {\xymatrixcolsep{4mm} \xymatrix{ \ar@{-||}[r]_{\tri\;\otimes} & }}
\newcommand{\lincomonwl} {\xymatrixcolsep{4mm} \xymatrix{ \ar@{-||}[r]_{\otimes\;\tri} & }}

\newcommand{\linbialgwl} {\xymatrixcolsep{4mm} \xymatrix{ \ar@{-||}[r]^{\otimes\;\tri}_{\otimes\;\tri} & }}
\newcommand{\linbialgwr} {\xymatrixcolsep{4mm} \xymatrix{ \ar@{-||}[r]^{\tri\;\otimes}_{\tri\;\otimes} & }}
\newcommand{\linbialgwb} {\xymatrixcolsep{4mm} \xymatrix{ \ar@{-||}[r]^{\circ}_{\bullet} & }}

\newcommand{\monoid}[1]{(#1, \mulmap{1.5}{white}: #1 \ox #1 \to #1, \unitmap{1.5}{white}: \yon \to #1)}
\newcommand{\comonoid}[1]{(#1, \comulmap{1.5}{white}: #1 \to #1 \ox #1, \counitmap{1.5}{white}: #1 \to \yon)}
\newcommand{\bialg}[1]{(#1, \mulmap{1.5}{white}, \unitmap{1.5}{white}, \comulmap{1.5}{black}, \counitmap{1.5}{black})}
\newcommand{\bialgb}[1]{(#1, \mulmap{1.5}{black}, \unitmap{1.5}{black}, \comulmap{1.5}{white}, \counitmap{1.5}{white})}

\newcommand{\mulmap}[2]{
	\begin{tikzpicture}[scale={#1}]
		\begin{pgfonlayer}{nodelayer}
			\node [style=circle, scale=0.4, fill={#2}] (5) at (0.32, 0.25) {};
			\node [style=none] (6) at (0.07, 0.5) {};
			\node [style=none] (7) at (0.57, 0.5) {};
			\node [style=none] (8) at (0.32, 0) {};
			\node [style=none] (9) at (0.64, 0.5) {};
		\end{pgfonlayer}
		\begin{pgfonlayer}{edgelayer}
			\draw [style=none] (8.center) to (5);
			\draw [style=none, bend left, looseness=1.25] (5) to (6.center);
			\draw [style=none, bend right, looseness=1.25] (5) to (7.center);
		\end{pgfonlayer}
	\end{tikzpicture}	
}

\newcommand{\unitmap}[2]{
\begin{tikzpicture}[scale=#1]
	\begin{pgfonlayer}{nodelayer}
		\node [style=circle, scale=0.4, fill=#2] (0) at (0, 0) {};
		\node [style=none] (1) at (0, -0.4) {};
		\node [style=none] (4) at (0.13, 0) {};
	\end{pgfonlayer}
	\begin{pgfonlayer}{edgelayer}
		\draw [style=none] (0) to (1.center);
	\end{pgfonlayer}
\end{tikzpicture} }

\newcommand{\comulmap}[2]{
	\begin{tikzpicture}[scale={#1}]
		\begin{pgfonlayer}{nodelayer}
			\node [style=circle, scale=0.4, fill={#2}] (5) at (0.32, 0.25) {};
			\node [style=none] (6) at (0.07, 0) {};
			\node [style=none] (7) at (0.57, 0) {};
			\node [style=none] (8) at (0.32, 0.5) {};
			\node [style=none] (9) at (0.64, 0) {};
		\end{pgfonlayer}
		\begin{pgfonlayer}{edgelayer}
			\draw [style=none] (8.center) to (5);
			\draw [style=none, bend right, looseness=1.25] (5) to (6.center);
			\draw [style=none, bend left, looseness=1.25] (5) to (7.center);
		\end{pgfonlayer}
	\end{tikzpicture}
}

\newcommand{\counitmap}[2]{
\begin{tikzpicture}[scale=#1, rotate=180]
	\begin{pgfonlayer}{nodelayer}
		\node [style=circle, scale=0.4, fill=#2] (0) at (0, 0) {};
		\node [style=none] (1) at (0, -0.4) {};
		\node [style=none] (4) at (0.13, 0) {};
	\end{pgfonlayer}
	\begin{pgfonlayer}{edgelayer}
		\draw [style=none] (0) to (1.center);
	\end{pgfonlayer}
\end{tikzpicture}
}

\newcommand{\biglens}[2]{
     \begin{bmatrix}{\vphantom{f_f^f}#2} \\ {\vphantom{f_f^f}#1} \end{bmatrix}
}
\newcommand{\littlelens}[2]{
     \begin{bsmallmatrix}{\vphantom{f}#2} \\ {\vphantom{f}#1} \end{bsmallmatrix}
}
\newcommand{\coclose}[2]{
  \relax\if@display
     \biglens{#2}{#1}
  \else
     \littlelens{#2}{#1}
  \fi
}

\newcommand{\qqand}{\qquad\textnormal{and}\qquad}

\author{David I.\ Spivak, Priyaa Varshinee Srinivasan}

\thanks{We thank Harrison Grodin and Reed Mullanix for their conjecture, that normal duoidal categories are linear distributive categories, which spawned the present paper. We thank Nathaniel Arkor for discussions on species and polynomial functors, and for sketching out the connection between them. We thank Rose Kudzman-Blais for her insights into the definition of duoidal categories.\\
This material is based upon work supported by the Air Force Office of Scientific Research under award number FA9550-23-1-0376. The second named author was partially funded by the Estonian Research Council (MOB3JD1227).\\
\indent\indent Published in the {\em Theory and Applications of Categories}, Vol. 45, No. 42, 2026, pp. 1748–1781.}

\address{Topos Institute\\
Berkeley, California, USA \\[5pt]
Tallinn University of Technology\\
Estonia 
}

\title{What kind of linearly distributive category do polynomial functors form?}

\copyrightyear{2026}



\eaddress{dspivak@gmail.com \CR priyaavarshinee@gmail.com}

\begin{document}
\allowdisplaybreaks

\maketitle

\begin{abstract}

This paper has two purposes. The first is to extend the theory of linearly distributive categories by considering the structures that emerge in a special case: the normal duoidal category $(\Poly,\yon,\otimes,\tri)$ of polynomial functors under Dirichlet and substitution product. This is an isomix LDC which is neither $*$-autonomous nor fully symmetric. The additional structures of interest here are a closure for $\otimes$ and a co-closure for $\tri$, making $\Poly$ a bi-closed LDC, which is a notion we introduce in this paper. 

The second purpose is to use $\Poly$ as a source of examples and intuition about various structures that can occur in the setting of LDCs, including duals, cores, linear monoids, and others, as well as how these generalize to the non-symmetric setting. To that end, we characterize the linearly dual objects in $\Poly$: every linear polynomial has a right dual which is a representable. It turns out that the linear and representable polynomials also form the left and right cores of $\Poly$. Finally, we provide examples of linear monoids, linear comonoids, and linear bialgebras in $\Poly$.
\end{abstract}
\setcounter{tocdepth}{1}
\tableofcontents

\section{Introduction}

Cockett and Seely introduced weakly distributive categories, renamed to linear distributive categories (LDCs), as a categorical semantics for multiplicative linear logic \cite{CS97}. Adding negation to the semantics produces $*$-autonomous categories. The primary contribution of this paper is to present a detailed study of a reasonably approachable yet a rich example of an LDC, one which is neither symmetric nor $*$-autonomous, namely the category $\Poly$ of polynomial functors and natural transformations. While $\Poly$ is a well-known category with a vast body of literature, to the best of our knowledge this paper is the first to explore $\Poly$ purely from the perspective of LDCs.

Indeed, though this project is focused on the category $\Poly$ as an LDC, our definitions and results are more general. We first prove a straightforward result, that normal duoidal categories are isomix LDCs---there is a faithful functor as we'll see in \cref{lemma.normalDuo_Iso}---thereby providing a family of examples for isomix LDCs, one of which is $\Poly$. These results are presented in \cref{Sec: duoidal}.

The category $\Poly$ is not $*$-autonomous: not every object in $\Poly$ has a dual. However, there are certain pairs of polynomials which are duals: for any set $A$, the linear polynomial $A \yon$ is left dual to the representable $\yon^A$, and these are the only polynomials with duals.%
\footnote{We will explain this notation in \cref{Sec: Poly prelim}.} 
We prove these results in more generality by introducing the notion of \emph{biclosed LDCs} which are LDCs with a closure for the $\ox$ operation and a coclosure for the $\tri$ operation, i.e.\ LDCs for which $(- \ox a)$ has a right adjoint $[a, -]$, and $(- \tri a)$ has a left adjoint $\coclose{-}{a}$. In $\Poly$, the former $[-,-]$ is the closure for the Dirichlet product and the latter $\coclose{-}{-}$ is the left Kan extension. We characterize the dual objects in biclosed LDCs, generalizing the above results for $\Poly$; all this is shown in \cref{Sec: Duals}.

Next we explore the notion of \textit{core} in mix LDCs within our context. Mix LDCs are equipped with a mix map $\m: \bot \to \top$, which induces a natural transformation $\indep_{a,b}: a \ox b \to a \tri b$ for all objects $a, b$. An isomix LDC is a mix LDC in which the mix map $\m: \bot \cong \top$ is an isomorphism. The core of a symmetric (iso)mix LDC is the full subcategory spanned by those objects for which the $\indep_{a,-}$ is a natural isomorphism. That is, an object $a$ is in the core if the map $\indep_{a,x}: a \ox x \cong a \tri x$ is an isomorphism, for all $x$; from this it follows from the symmetry that $\indep_{x,a}: x \ox a \cong x \tri a$ is also an isomorphism. 

The notion of core in mix LDCs has been studied only for symmetric LDCs. Inspired by $\Poly$, this paper explores the notion of core in non-symmetric settings. Without the assumption of symmetry, one may define the left core (resp.\ right core) to be the full subcategory spanned by the objects $a$ for which $\indep_{a,-}$ (resp.\ $\indep_{-,a}$) is a natural isomorphism. We show that in $\Poly$, the linear polynomials ($A\yon$) comprise the left core and the representables ($\yon^A$) comprise the right core. The cores and the duals thus coincide: a polynomial is in the left core iff it is a left dual iff it is linear, and it is in the right core iff it is a right dual iff it is representable.  

For a mix LDC $\X$, we say that the left and the right cores are \emph{opposing} when $\lCore(\X) \cong \rCore(\X)^{\op}$. We will show that $\Poly$ has opposing cores, which are equivalent to $\Set$ and $\Set^\op$, respectively. Another example of mix LDCs with opposing cores are compact closed categories.  \cref{Sec:Core} provides the above definitions and proves relevant results. 

Linear monoids and linear comonoids generalize Frobenius algebras from monoidal categories to LDCs. We show that $\Poly$ has non-trivial linear-monoid-\emph{like} structures: every monoid in $\Set$ defines a \emph{left} linear monoid and every set induces a \emph{right} linear monoid in $\Poly$; these definitions are new. Such linear monoids and linear comonoids in $\Poly$ interact to produce left and right linear bialgebras. \cref{Sec.linmon} contains these results.

\paragraph{Contributions.}
The following are the main technical contributions of this paper: 

\begin{enumerate}[(1), nosep]
    \item We prove that all normal duoidal categories are isomix LDCs. In particular, the category $\Poly$ is an isomix LDC. 
    \item We define the notion of biclosed LDCs.
    \item We characterize the linear duals in isomix biclosed LDCs. 
    \item We prove that a polynomial $q$ is left dual of $p$, i.e.\ $q\dashvv p$ iff $q = A \yon$ and $p = \yon^A$ for some set $A$.  
    \item We define the left and the right cores of non-symmetric LDCs. In $\Poly$, we show that the left core consists of the linear polynomials ($A\yon$) and the right core consists of the representable polynomials ($\yon^A$).
    \item We show that for every monoid in $\Set$, we get a left linear monoid and a right linear comonoid in $\Poly$. Moreover, every set induces a left linear comonoid and a right linear monoid in $\Poly$. These structures interact to produce right and left linear bialgebras which are the notion we introduce in this paper. 
\end{enumerate}

In \cref{sec.prelim} we recollect the preliminary definitions and results used in this paper. 
 
\section{Preliminaries}\label{sec.prelim}

For preliminaries on $\Poly$ see \cref{Sec: Poly prelim}; for preliminaries on LDCs see \cref{Sec: LDC prelim}.

\subsection{The category {\sf Poly}}
\label{Sec: Poly prelim}

While there are many equivalent definitions of polynomial functors \cite{NiS23},%
\footnote{For a friendly and very non-technical introduction to polynomial functors, see \url{https://topos.site/blog/2022-01-19-poly-makes-me-happy/} or \url{https://topos.site/blog/2023-11-13-solving-problem-solving/}.}
in this paper we view them as coproducts of representables. 

Given an object $A : \C$, a functor $F: \C \to \Set$ is {\bf represented by $A$} if there is an isomorphism $F \cong \C(A, -)$; in this case we also say that $F$ is {\bf representable}. For functors $\Set\to\Set$, we denote the functor $\Set(A,-)$ by $\yon^A$.

\begin{definition}[{\cite[Defintion 2.1]{NiS23}}]
A {\bf polynomial functor} is a functor $p: \Set \to \Set$ that is isomorphic to a coproduct of representables, i.e.\
\[ p \cong \sum_{i : I} \yon^{A_i} \]
for some indexing set $I:\Set$ and sets $A_i:\Set$. We refer to the elements of $I$ as the {\bf positions} of $p$, and we refer to elements of $A_i$ as the {\bf directions} of $p$ at position $i:I$. 
\end{definition}

The following are a few examples of polynomial functors:

\begin{enumerate}[(i), nosep]
\item $\yon^A$ is a representable polynomial, for any $A:\Set$. The terminal object of $\Poly$ is $\yon^0$; we denote this by $1$.
\item $\sum_{i:I} \yon^0$ is a constant polynomial; it sends $X\mapsto I$ for all $X:\Set$. We denote this by $I\coloneqq\sum_{i:I}1$.
\item $\sum_{i : I} \yon \cong I \yon$ is a linear polynomial, for any $I:\Set$.
\item $\yon^3 + \yon^2$ where $2$ and $3$ are 2-element and 3-element sets respectively. For this polynomial, the index set $I = 2$, $A_{1} = 3$, and $A_2 = 2$. 
\end{enumerate}

Corolla forests are an intuitive visualization tool for polynomials \cite[Chapter 2]{NiS23}. 
For example, the polynomial $\yon^3 + \yon^2$, can be drawn as a corolla forest as follows:
\[  \includegraphics[scale=0.05]{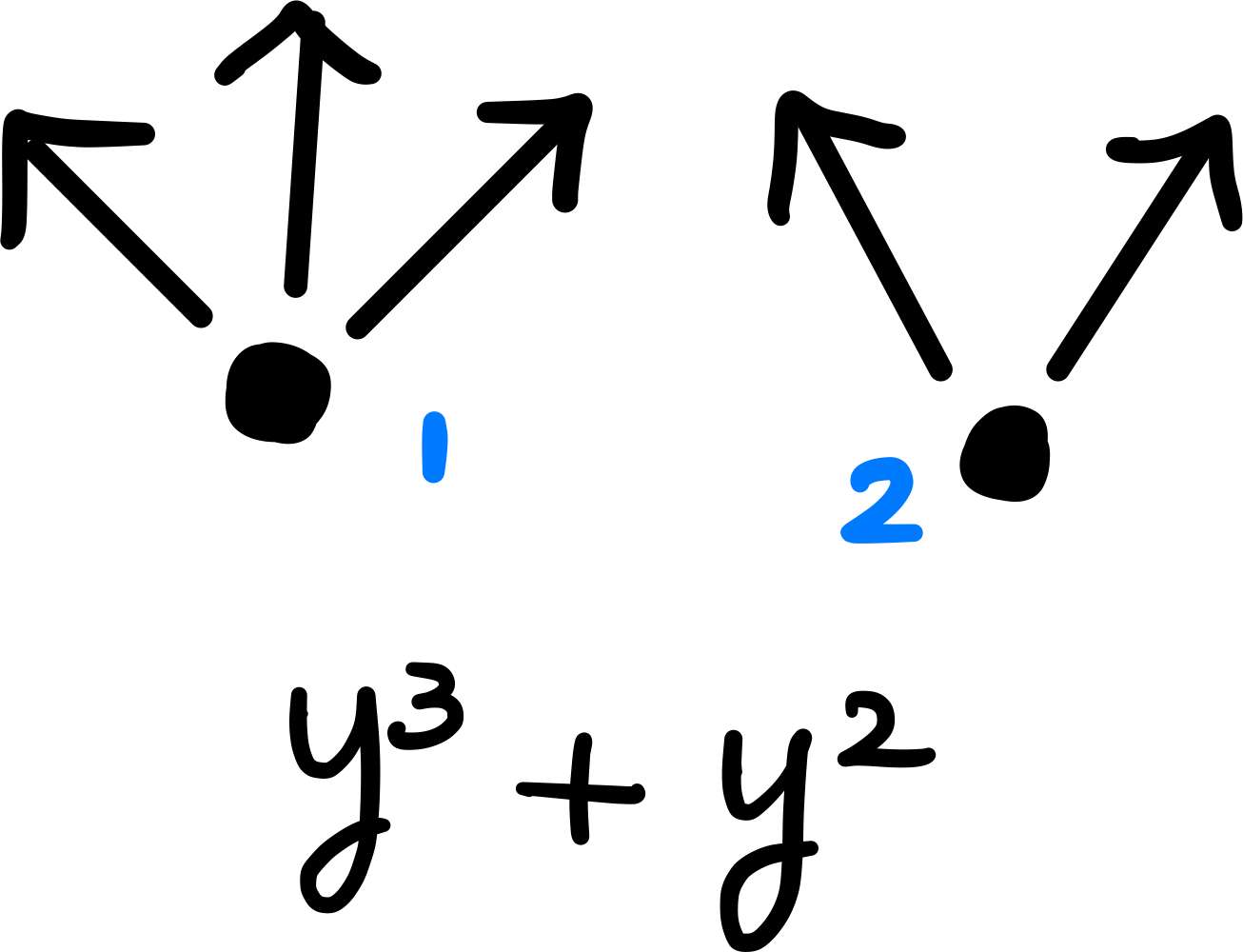} \]

For a polynomial $p \cong \sum_{i : I} \yon^{A_i}$, the elements of $I$ are referred to as the {\bf positions} of $p$. Thus, each corolla in a corolla forest corresponds to a position of its polynomial. For each $i:I$, the elements of $A_i$ are referred to as the {\bf directions at that position}. The directions at position $i$ are denoted as the branches of the corolla corresponding to that position.

The positions of a polynomial $p$ can be found by evaluating the polynomial functor at the single element set $1$: 
\[ p(1) = \sum_{i:I} \yon^{A_i} (1) \cong \sum_{i:I} 1^{A_i} \cong \sum_{i:I} 1 \cong I \]
Hence, writing $p[P]\coloneqq A_P$ for any position $P:p(1)$, we can denote a polynomial succinctly as follows: 
\begin{equation}\label{eqn.succinct_poly}
 p \cong \sum_{P:p(1)} \yon^{p[P]}
 \end{equation}

\begin{definition}
We define $\Poly$ to be the category of polynomial functors $\Set\to\Set$ and the natural transformations between them.
\end{definition}

We can give a fully set-theoretic account of the natural transformations between any two polynomial functors by using the universal property of coproducts and the Yoneda Lemma. Indeed, for any two polynomials $p = \sum_{P: p(1)} \yon^{p[P]}$ and $q = \sum_{Q: q(1)} \yon^{q[Q]}$,
\begin{align}
\nonumber
\Poly(p,q) &= \Poly\left(\sum_{i:P} \yon^{p[P]}, q\right) \\\nonumber
&\cong \prod_{P:p(1)} \Poly\left(\yon^{p[P]}, q\right) \quad \quad \text{(Universal property of coproduct)}\\ \nonumber
&\cong \prod_{P:p(1)} q(p[P])  \quad \quad  \text{(Yoneda Lemma)} \\  \nonumber
& \cong \prod_{P:p(1)} \sum_{Q: q(1)} p[P]^{q[Q]} \\\label{eqn.poly_map}
& = \prod_{P:p(1)} \sum_{Q: q(1)} \prod_{e: q[Q]} \sum_{d:p[P]} 1 
\end{align}
Since products distribute over sums, the expression in \eqref{eqn.poly_map} tells us that a natural transformation $\varphi: p \to q$ can be identified with:  
\begin{itemize}[nosep]
\item a function $\varphi_1: p(1) \to q(1)$, from the positions of $p$ to the the positions of $q$, and
\item for each $P:p(1)$, a function $\varphi_P^\sharp : q[\varphi_1(P)] \to p[P]$, from the directions of $q$ at $\varphi_1(P)$ to directions of $p$ at $P$.
\end{itemize}

A map $\varphi$ of polynomials is called {\bf cartesian} if, for each $P:p(1)$, the map $\varphi_P^\sharp$ is a bijection.%
\footnote{In general, a natural transformation $\varphi$ is called \emph{cartesian} if its naturality squares are pullbacks. A map $\varphi$ of polynomials is cartesian in this sense iff it satisfies the above condition, that $\varphi_P^\sharp$ is a bijection.}

As an example of polynomial maps, let us consider the polynomials $p = \yon^3 + \yon^2$ and $q = \yon + \yon^2$. The following diagram depicts a natural transformation $\varphi: p \to q$: 
\[  \includegraphics[scale=0.12]{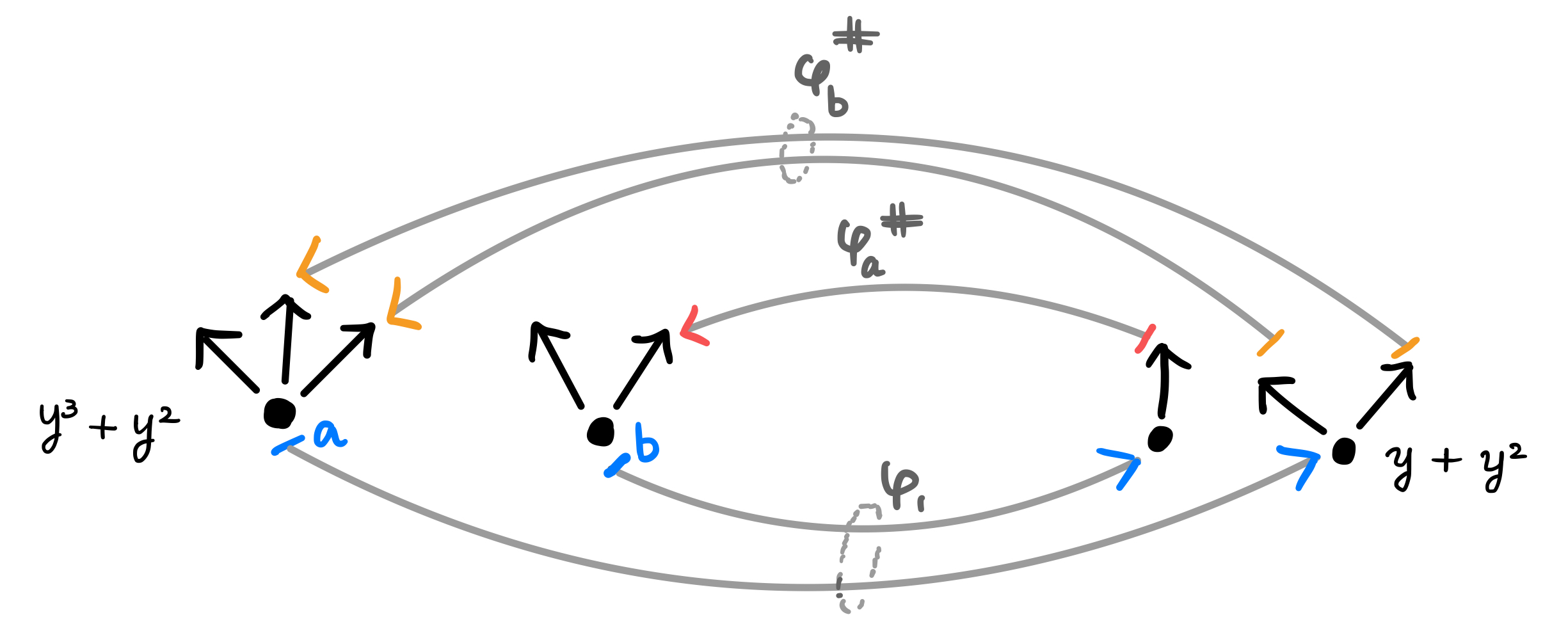} \]
One can count that there are $(3^1+3^2)\times(2^1+2^2)=72$ natural transformations $p\to q$.

Here are two useful functors between $\Poly$ and $\Set$, based on the positions and the directions of polynomials:
\begin{itemize}
\item the functor $\Poly(\yon, -): \Poly \to \Set$, which we denote by $-(1)$. For any polynomial $p$, an element of $p(1)$ is a position of $p$. 
\item the functor $\Poly(-,\yon): \Poly \to \Set^\op$, which we denote by $\Gamma_-$. For any polynomial $p$, an element of $\Gamma_p$ is a \emph{global section}: it picks a direction for each position of $p$: 
\begin{equation}\label{eqn.gamma}
	\Gamma_p \cong \prod_{P: p(1)} p[P]
\end{equation}
\end{itemize}

While $\Poly$ has many monoidal structures, in this article we are concerned with only two: the tensor (Dirichlet) product and the substitution product \cite{Spi22}. The tensor product $\ox$ is given by the Day convolution of Cartesian product $\times$ in $\Set$, and the substitution product $\tri$ is given by functor composition. They have straightforward formulas:
\begin{align}\label{eqn.dirichlet}
p \ox q = \sum_{P: p(1)} \yon^{p[P]} \ox \sum_{Q: q(1)} \yon^{q[Q]} &:= \sum_{(P,Q): p(1) \times q(1)} \yon^{p[P] \times q[Q]} \\\nonumber
p \tri q = \sum_{P: p(1)} \yon^{p[P]} \tri \sum_{Q: q(1)} \yon^{q[Q]} &:=  
 \sum_{P: p(1)} \Big(\sum_{Q: q(1)} \yon^{q[Q]} \Big)^{p[P]} 
 = \sum_{P:p(1)} \prod_{d:p[P]} \sum_{Q:q(1)} \prod_{e:q[Q]} \yon 
\end{align} 

We leave the proofs of the following statements as exercises for the reader.

\begin{lemma} 
\label[lemma]{lemma.tensored_with_1}
For any polynomial $p$, its set of positions p(1) is isomorphic to $1 \otimes p$ (and $p \otimes 1$).
\end{lemma}

\begin{proposition}\label[proposition]{prop.lin_rep_strong_monoidal}~
\begin{enumerate}[a., nosep]
	\item The functor $A\mapsto A\yon$ is strong monoidal $(\Set,1,\times)\to(\Poly,\yon,\otimes)$.
	\item The functor $p\mapsto p(1)$ is strong monoidal $(\Poly,\yon,\otimes)\to(\Set,1,\times)$.
	\item The functor $A\mapsto\yon^A$ is strong monoidal $(\Set^\op,1,\times)\to(\Poly,\yon,\otimes)$.
	\item The functor $p\mapsto\Gamma_p$ is colax monoidal $(\Poly,\yon,\tri)\to(\Set^\op,1,\times)$.
\end{enumerate}
\end{proposition}

\begin{remark}
 \cref{prop.lin_rep_strong_monoidal}-$(a)$ is a special case of \cite[Example 24]{GaF16} once $\Poly$ is proven to be a normal duoidal category, see \cref{lemma.poly_normal_duoidal}.
\end{remark}

\subsection{Linearly distributive categories (LDCs)}
\label{Sec: LDC prelim}

Here we discuss various flavors of LDCs and the appropriate sorts of functors between them. 

\subsubsection{Flavors of LDCs}

A {\bf linearly distributive category (LDC)} is a category $\X$, equipped with two monoidal structures\footnote{The operation $a \tri b$ is usually written as $a \oa b$ in the literature.}
$$(\X, \ox, \top)~~~\mbox{ and } ~~~ (\X, \tri, \bot)$$
linked by natural transformations
\begin{equation}\label{eqn.lin_dist}
\begin{aligned}
&\partial^L_L: a \ox (b \tri c) \rightarrow  (a \ox b) \tri c \\
& \partial^R_R: (b \tri c) \ox a \rightarrow b \tri (c \ox a)
\end{aligned} 
\end{equation}
called  \emph{linear distributors},%
\footnote{The maps $\partial^L$ and $\partial^R$ are called \emph{linear}  distributors because they do not copy $a$ as in distributive categories: 
\[ a \times (b + c)  \cong (a \times b) + (a \times c) \]
In this article, we use the word \emph{linear} for two different but firmly-established notions. First, a polynomial is \emph{linear} if it has the form $A\yon$ for some set $A$. Second, the terminology of linear distributive categories often appends the term ``linear'' to refer to various notions in that context, e.g.\ \emph{linear duals}. Hopefully, this name-space collision will not cause confusion.
}
 such that the the associators and unitors of the $\ox$ and the $\tri$ products interact coherently with the linear distributors \cite{BCST96, CS97}.

\begin{example}
Every monoidal category is an LDC in which the two tensor products coincide. 
\end{example}
For more examples, see \cite[Chapter 2]{Sri21}.

A {\bf symmetric LDC} is an LDC in which both monoidal structures are symmetric,  i.e.\ having symmetry maps $\sigma_\ox$ and $\sigma_\tri$ such that the following diagram commutes:
\[ \begin{tikzcd}
(a \tri b) \ox c \ar[d, "\partial^R_R"'] \ar[r, "\sigma_\ox" ]
	& c \ox (a \tri b) \ar[r, "\sigma_\tri"]  
	& c \ox (b \tri a) \ar[d, "\partial^L_L"] \\ 
a \tri (b \ox c) 
	& a \tri (c \ox b)  \ar[l, "\sigma_\tri"]
	& a \tri (b \ox c) \ar[l, "\sigma_\ox"] 
\end{tikzcd}\]
In a symmetric LDC, the symmetry maps induce the following \emph{permuting} distributivity maps:
\begin{equation}\label{eqn.other_dists}
\begin{aligned}
     \partial^L_R &: a \otimes (b \tri c) \to b \tri (a \otimes c) \\ 
     \partial^R_L &: (b \tri c) \otimes a \to (b \otimes a) \tri c
\end{aligned}
\end{equation}
An LDC which is equipped with all the four linear distributivity maps is said to be {\bf non-planar} \cite{CS97}.

Motivated by $\Poly$, we consider LDCs in which only the tensor $\ox$ is symmetric. 
\begin{definition}
A {\bf $\otimes$-symmetric LDC} $(\X, \otimes, \top, \tri, \bot)$ is an LDC such that:
\begin{enumerate}[nosep]\item $(\X, \otimes, \top)$ is a symmetric monoidal category (\emph{parallel}  structure), and
\item $(\X, \tri, \bot)$ is a (not-necessarily-symmetric) monoidal category (\emph{sequential} structure).
\end{enumerate}
\end{definition}

\begin{lemma}
\label[lemma]{lemma.planar} 
All $\ox$-symmetric LDCs are non-planar.
\end{lemma}
\begin{proof}[Sketch]
In a $\otimes$-symmetric LDC, the symmetry of the tensor product $\otimes$ induces the two permuting distributivity maps as follows: 
\[ \partial^L_R := \left( a \otimes (b \tri c) \to^{\sigma_\ox} (b \tri c) \ox a \to^{\partial^R_R} b \tri (c \ox a) \to^{b \tri \sigma_\ox} b \tri (a \ox c) \right) \] 
The other permuting map $\partial^R_L$, is constructed similarly.
\end{proof}

An LDC satisfies the mix inference rule in linear logic if it has a map $\bot \to \top$ satisfying certain coherences:

\begin{definition}[\cite{CS97a}]
A {\bf mix LDC} is an LDC with a map $\m: \bot \to \top$, called the {\bf mix map}. The mix map induces, for any $a,b:\X$, a natural map $\indep_{a,b}: a \ox b \to a \tri b$, as shown as the dotted arrow in the following commutative diagram:
\begin{equation}
\label{eqn: mix coherence}
\begin{tikzcd}[column sep=50pt, ampersand replacement=\&]
a \ox b \ar[r, "\id \ox u_\tri^{L^{-1}}"] \ar[d, "(u_\tri^R)^{-1}\ox \id"'] \ar[ddrr, dotted, "\indep_{a,b}"] \& 
a \ox (\bot \tri b) \ar[r, "\id \ox (\m \tri \id)"] \& 
a \ox ( \top \tri b) \ar[d, "\partial^L"] \\
(a \tri \bot) \ox b \ar[d, "\partial^R"'] \&  
\& 
( a \ox \top ) \tri b  \ar[d, "u_\ox^R \tri \id"] \\
a \tri (\bot \ox b) \ar[r, "\id \tri (\m \ox \id)"'] \& 
a \tri (\top \ox b) \ar[r, "\id \tri u_\ox^L"'] \&  
a \tri b
\end{tikzcd}
\end{equation}
An {\bf isomix LDC} is a mix LDC in which the mix map $\top \to^{\m}_{\raisebox{1pt}{$\cong$}} \bot$ is an isomorphism. 

\end{definition}

When $\m: \bot \to \top$ is an isomorphism, the coherence requirement \eqref{eqn: mix coherence} is automatically satisfied (see \cite[Lemma 6.6]{CS97a}).

In the LDC literature, the $\indep$ map is usually referred to as the {\bf mixor} and is written as {\sf mx}. In this article, we use {\bf indep}, because of its semantics in $\Poly$; see below \eqref{eq.indep}.

A {\bf compact LDC} is an isomix LDC in which for all $a,b$, the map $\indep_{a,b}: a \ox b \to^{\cong} a \tri b$ is an isomorphism. In fact, any compact LDC is linearly equivalent to a monoidal category \cite[Corollary 2.18]{Sri21}. The term `compact' in the context of LDCs refers to the fact that the two monoidal products are isomorphic via the $\indep$ map. 

A {\bf monoidal category} is a compact LDC in which for all $a,b$, the map $\indep_{a,b} = \id_{a \ox b}$ is the identity natural transformation. 

The following schematic diagram summarizes some relevant flavors of LDCs:
\[
	\begin{tikzpicture}[scale=1.8]
		\begin{pgfonlayer}{nodelayer}
			\node [style=circle, scale=2, color=black, fill=red] (0) at (-5.75, 2.75) {};
			\node [style=circle, scale=2, color=black, fill=red!70] (1) at (-3.5, 2.75) {};
			\node [style=circle, scale=2, color=black, fill=red!60] (2) at (-1, 2.75) {};
			\node [style=circle, scale=2, color=black, fill=red!40] (3) at (1.75, 2.75) {};
			\node [style=circle, scale=2, color=black, fill=red!20] (5) at (4, 2.75) {};
			\node [style=none] (4) at (-7.75, 2.75) {};
			\node [style=none] (6) at (6, 2.75) {};
			\node [style=none] (7) at (-5.75, 2) {LDC};
			\node [style=none] (8) at (-3.5, 4) {Mix LDC};
			\node [style=none] (9) at (-3.5, 3.5) {$\m: \bot \to \top$};
			\node [style=none] (10) at (-1, 2) {Isomix LDC};
			\node [style=none] (11) at (1.75, 4.25) {Compact LDC};
			\node [style=none] (12) at (1.75, 3.65) {$a \ox b \to^{\indep}_{\cong} a \tri b$};
			\node [style=none] (13) at (4, 2) {Monoidal category};
			\node [style=none] (14) at (-1, 1.4) {$\bot \to^{\m}_{\cong} \top$};
			\node [style=none] (15) at (4, 1.5) {$\ox = \tri$, $\top=\bot$};
			\node [style=none] (16) at (-5.75, 1.5) {$(\X, \ox, \top)$};
			\node [style=none] (17) at (-5.75, 1) {$(\X, \tri, \bot)$};
		\end{pgfonlayer}
		\begin{pgfonlayer}{edgelayer}
			\draw [dotted] (4.center) to (0);
			\draw (0) to (1);
			\draw (1) to (2);
			\draw (2) to (3);
			\draw (3) to (5);
			\draw [dotted] (5) to (6.center);
		\end{pgfonlayer}
	\end{tikzpicture}
\]
We will be focused mainly on the center dot: we will see that $\Poly$ is an isomix LDC.

\paragraph{Notation.} In this article, we write $\partial^L$ and $\partial^R$ for $\partial^L_L$ and $\partial^R_R$ respectively, and we will have very little use for the other two maps in \eqref{eqn.other_dists}. For the rest of the paper, the term LDC will refer to non-symmetric LDCs unless otherwise specified.

\subsubsection{Functors between LDCs}

Linear functors \cite{CS99} are the most general notion of maps between LDCs. In this article we use a simpler notion: that of Frobenius functors \cite{DP08}.

\begin{definition}[{\cite[Lemma 2.20]{Sri21}}]
    Suppose $\X$ and $\Y$ are LDCs. A {\bf Frobenius functor} consists of a functor $F: \X \to \Y$ equipped with: 
    \begin{enumerate}[(i), nosep]
        \item a lax monoidal structure $(F, m_\ox, m_I): (\X, \ox, \top) \to (\Y, \otimes, \top)$  and
        \item a colax monoidal structure $(F, n_\tri, n_I): (\X, \tri, \bot) \to (\Y, \tri, \bot)$
    \end{enumerate}
   such that  the laxors $m_\ox$ and $n_\tri$ and the distributivity maps interact as follows:
    \[ {\bf [F.1]}\quad  \begin{tikzcd}[column sep=large]
        F(a) \otimes F(b \tri c) \ar[r, "\id \otimes n_\tri"] \ar[d, "m_\otimes"'] & 
        F(a) \otimes (F(b) \tri F(c)) \ar[d, "\partial^L"] \\ 
        F(a \otimes (b \tri c)) \ar[d, "F(\partial^L)"'] 
        & (F(a) \otimes F(b)) \tri F(c) \ar[d, "m_\otimes \tri \id"] \\ 
        F((a \ox b) \tri c) \ar[r, "n_\tri"'] & F(a \ox b) \tri F(c)
    \end{tikzcd} \] \[
    \quad {\bf [F.2]}\quad \begin{tikzcd}[column sep=large]
        F(a \tri b) \otimes F(c) \ar[r, "n_\tri \otimes  \id"]  \ar[d, "m_\otimes"']  &
        (F(a) \tri F(b)) \otimes F(c) \ar[d, "\partial^R"] \\ 
        F((a \tri b) \ox c) \ar[d, "F(\partial^R)"'] & F(a) \tri (F(b) \otimes F(c)) \ar[d, "\id \tri m_\ox"] \\ 
        F(a \tri (b \otimes c)) \ar[r, "n_\tri"'] & F(a) \tri F(b \otimes c)
    \end{tikzcd} \]
\end{definition}

\begin{definition}[{\cite[Definition 2.22]{Sri21}}]
Suppose $\X$ and $\Y$ are mix LDCs. A {\bf mix functor} $F: \X \to \Y$ is a Frobenius functor making the following diagram commute: 
\[
\begin{tikzcd}
    F(\bot) \ar[r, "n_I"] \ar[d, "F(\m)"'] & 
    \bot \ar[d, "\m"] \\ 
    F(\top) \ar[r, <-, "m_I"'] & 
    \top 
\end{tikzcd}
\]
A mix functor $F$ is {\bf isomix} if $m_I: \top \to F(\top)$ and $n_I: F(\bot) \to \bot$ are isomorphisms. 
\end{definition}

\section{{\sf Poly} is an isomix LDC}
\label{Sec: duoidal}

In this section, we prove that $\Poly$ is an isomix LDC. This result follows from the general observation that there is a faithful functor from the category of normal duoidal categories and functors to that of (non-planar) isomix LDCs and isomix Frobenius functors.  

\begin{definition}[{\cite[Definition 6.1.]{AMa10}}]
    A {\bf duoidal category} is a category $\X$ with two monoidal structures $(\X, \ox, \top)$ and $(\X, \tri, \bot)$ along with a natural transformation:
    \[ \duo: ~(a \tri b) \ox (c \tri d) \to (a \ox c) \tri (b \ox d) \]
    called the {\bf interchange law}, and morphisms: 
        \begin{align*}
	    	e_\top:& ~\top \to \top \tri \top\\ 
            e_\bot:& ~\bot \otimes \bot \to \bot\\ 
            k:& ~\top \to \bot
        \end{align*}
    such that the functors $\tri: \X \times \X \to \X$ and $\bot: \mathbbm{1} \to \X$ are $\ox$-lax monoidal, and the assosciator and unitor natural isomorphisms of $(\tri, \bot)$ are $\ox$-monoidal natural transformations.     
    From the above data, it can be proven that the following squares commute \cite[Proposition 6.9.]{AMa10}: 
\[\begin{tikzcd}
	{(\top \tri \bot) \ox (\bot \tri \top)} & {\top \otimes \top} & \top & {\top \otimes \top} & {(\bot \tri \top) \otimes (\top \tri \bot)} \\
	{(\top \otimes \bot) \tri (\bot \otimes \top)} & {\bot \tri \bot} & \bot & {\bot \tri \bot} & {(\bot \otimes\top) \tri (\top \otimes\bot)}
	\arrow["\duo"', from=1-1, to=2-1]
	\arrow[""{name=0, anchor=center, inner sep=0}, "\cong"', from=1-2, to=1-1]
	\arrow["\cong"', from=1-3, to=1-2]
	\arrow["\cong", from=1-3, to=1-4]
	\arrow["k", from=1-3, to=2-3]
	\arrow[""{name=1, anchor=center, inner sep=0}, "\cong", from=1-4, to=1-5]
	\arrow["\duo", from=1-5, to=2-5]
	\arrow[""{name=2, anchor=center, inner sep=0}, "\cong"', from=2-1, to=2-2]
	\arrow["\cong"', from=2-2, to=2-3]
	\arrow["\cong", from=2-4, to=2-3]
	\arrow[""{name=3, anchor=center, inner sep=0}, "\cong", from=2-5, to=2-4]
	\arrow["{=}"', draw=none, from=0, to=2]
	\arrow["{=}"{description}, draw=none, from=1, to=3]
\end{tikzcd}\]

	We say that a duoidal category is {\bf normal} if the above composite $k$ is an isomorphism.
\end{definition}

\begin{definition}[{\cite[Definition 6.50.]{AMa10}}]
    A {\bf bilax duoidal functor} $F: (\X, \ox, \top, \tri, \bot) \to (\Y, \ox, \top, \tri, \bot)$ consists of a functor $F: \X \to \Y$ that is $\ox$-lax monoidal, $\tri$-colax monoidal, and such that the following diagrams commute:     
    \[
    {\bf [D.1]} \xymatrix{
        F((a \tri b) \otimes (c \tri d)) \ar[d]_{F(\duo)} &F(a \tri b) \otimes F(c \tri d) \ar[r]^{\hspace{-20pt}n_\tri \otimes n_\tri} \ar[l]_{m_\otimes} & (F(a) \tri F(b)) \otimes (F(c) \tri F(d)) \ar[d]^{\duo} \\ 
        F((a \otimes c) \tri (b \otimes d)) \ar[r]_{n_\tri} &F(a \otimes c) \tri F(b \otimes d) & (F(a) \otimes F(c)) \tri (F(b) \otimes F(d)) \ar[l]^{\hspace{-20pt}m_\otimes \tri m_\otimes} 
    }
    \]
    \[
    {\bf [D.2]} \xymatrix@C=1.5em{
        \top \ar[r]^{m_I} \ar[d]_{e_\top} & 
        F(\top) \ar[r]^{F(e_\top)} & 
        F(\top \tri \top) \ar[d]^{n_\tri} \\
        \top \tri \top \ar[rr]_{m_I \tri m_I} & &
        F(\top) \tri F(\top) 
    } ~~~~~~~~
    {\bf [D.3]} \xymatrix@C=1.5em{
        F(\bot) \otimes F(\bot) \ar[d]_{m_\otimes} \ar[rr]^{n_I \otimes n_I} &  & \bot \otimes \bot \ar[d]^{e_\bot} \\
        F(\bot \otimes \bot) \ar[r]_{F(e_\bot)} & F(\bot) \ar[r]_{n_I} &  \bot 
    }
    \]
    \[ {\bf [D.4]}~~ 
    \begin{tikzcd}
     \top \ar[r, "m_I"] \ar[d, "k"' ] & F(\top) \ar[d, "F(k)"] \\
     \bot \ar[r, <- , "n_I"'] & F(\bot)  
    \end{tikzcd}
    \]

A bilax duoidal functor $F$ between normal duoidal categories is {\bf normal} if the unit laxors are isomorphisms:
\[
m_I: \top \to^{\cong} F(\top) ~~~~ \text{ and } ~~~~~ n_I: F(\bot) \to^{\cong} \bot
\]

\end{definition}

\begin{figure}
\centering
\includegraphics[scale=0.8]{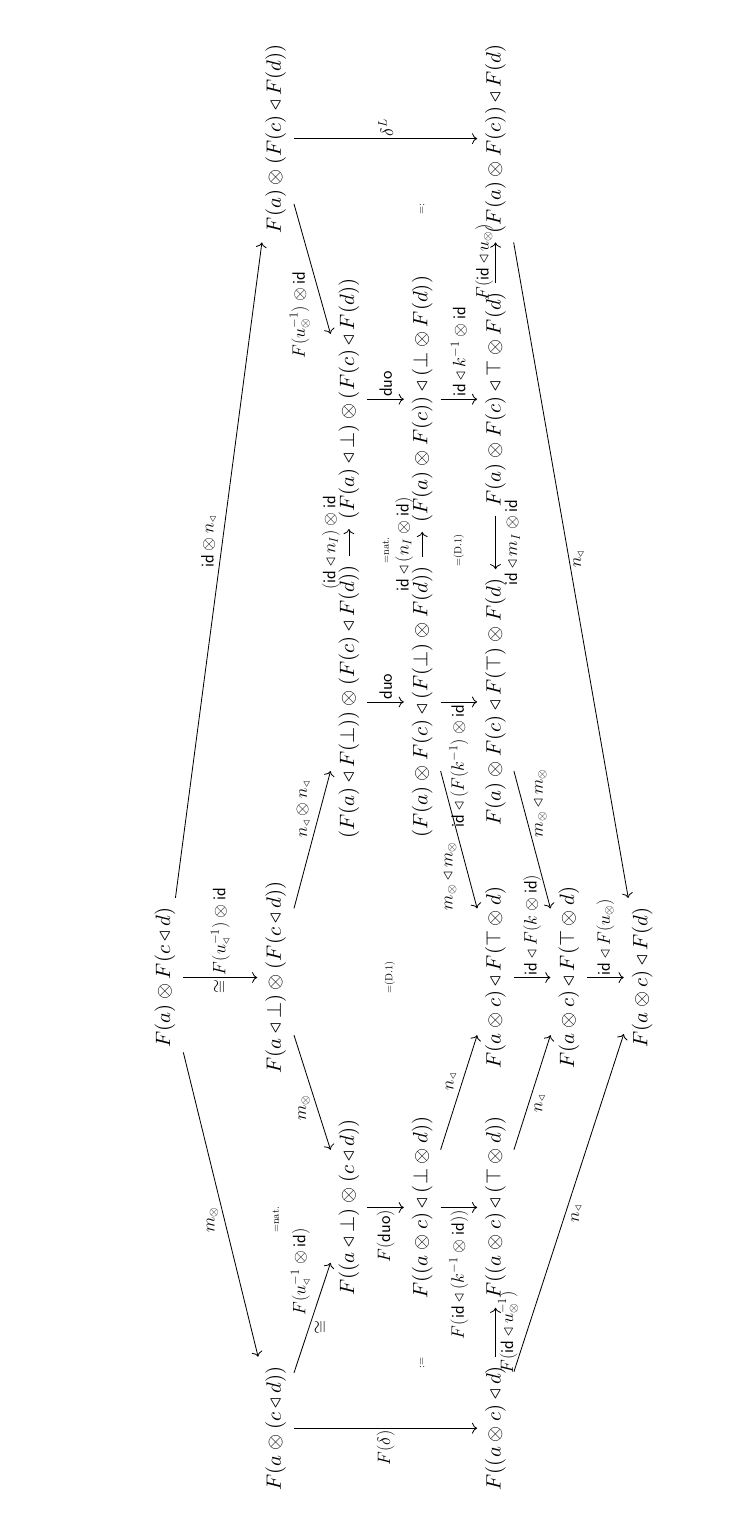}
\caption{Normal duoidal functor is also isomix Frobenius}
\label[figure]{Fig: commuting diagram}
\end{figure}

Let $\nD$ be the category of normal duoidal categories and bilax normal duoidal functors. Let $\Iso$ be the category of isomix categories and isomix functors. 

\begin{lemma}
\label[lemma]{lemma.normalDuo_Iso}
    There is a faithful functor $H: \nD \to \Iso$.
\end{lemma}

\begin{proof}[Sketch]
Suppose $(\X, \ox, \top, \tri, \bot)$ is a normal duoidal category. We show that it can be given an isomix LDC structure by defining the linear distributivity and the mix maps. The linear distributivity maps are constructed from the duoidal map and the isomorphism $\top \to^{k}_{\cong} \bot$ by introducing and eliminating monoidal units. The construction of the left distributor $\partial^L$ is shown in \eqref{eqn: left distributor}.
\begin{equation}
\label{eqn: left distributor}
 \begin{tikzcd}[column sep = large]
a \ox (b \tri c) 
 	\ar[rr, "{u_\tri^{-1} \ox \id}", "\cong"' ]
	\ar[d, "\partial^L"']
	\ar[drr, phantom, "\partial^L:="]
 && (a \tri \bot) \ox (b \tri c)
 	\ar[d, "\duo" ]
\\ 
 (a \ox b) \tri c 
& (\top \ox b) \tri (a \ox c)
	\ar[l, "\id \tri u_\ox"', "\cong" ]
& (a \ox b) \tri (\bot \ox c)	
	\ar[l, "k^{-1}"', "\cong" ]
\end{tikzcd} 
\end{equation}
The {\bf mix map} $\m: \bot \to \top$ is defined to be $k^{-1}$. The right distributivity map $\partial^R$ and the two permuting distributivity maps $\partial^L_R$, $\partial^R_L$ are constructed similarly. 
%

Suppose $\X$ and $\Y$ are normal duoidal categories, and $F: \X \to \Y$ is a normal duoidal functor. Then,  $H(F) := F$ is a Frobenius functor, see \cref{Fig: commuting diagram} for the proof that [F.1] holds. In \cref{Fig: commuting diagram}, the inner squares shown without label commute because $F$ is a lax monoidal functor or a co(lax)monoidal functor. Similary, it can be proven that [F.2] holds. 

The normal duoidal functor $F$ is also an isomix functor because $\top \to^{k} \bot$ is an isomorphism, and {\bf [D.4]} holds. 
\end{proof}

The category of polynomial functors and natural transformations is normal duoidal:

\begin{lemma}[\cite{Spi22}] 
\label[lemma]{lemma.poly_normal_duoidal}
The category $({\sf Poly}, \otimes, \tri, \yon)$ is normal duoidal.
\end{lemma}
\begin{proof} [Sketch of proof] 
For any symmetric monoidal structure $\cdot$ on $\Set$, there is a symmetric monoidal structure $\odot$ on {\sf Poly} which is the Day convolution of $\cdot$. The Cartesian product on {\sf Set} induces the $\ox$ structure on {\sf Poly} given in \cref{eqn.dirichlet}. 
Thus, for $p,q: \Poly$, and  $A,B: \Set$:
\[ p(A) \times q(B) \to^{\sf Day} (p \otimes q) (A \times B) \]
natural in $A,B : \Set$, where $p(A) := p \tri A$ denotes functor application. The above map satisfies the following universal property: for all functors $r\colon\Set\to\Set$ and functions $\phi: p(A) \times q(B) \to r(A \times B)$, there exists a unique (dotted) map which makes the following diagram commute: 
\begin{equation}
\label{eq.dayUniversalProp}
\xymatrix{
p(A) \times q(B) \ar[r]^{\sf Day} \ar[dr]_{\phi} & (p \otimes q) (A \times B) \ar@{.>}[d] \\
& r(A \times B)
}
\end{equation}
With this data, one can prove that $({\sf Poly}, \otimes, y, \tri, y)$ is normal duoidal. Indeed, to give the duoidal map $d$ for $p_1, p_2, q_1, q_2: {\sf Poly}$,
\[ d: (p_1 \tri p_2) \otimes (q_1 \tri q_2) \to (p_1 \otimes q_1) \tri (p_2 \otimes q_2) \]
it suffices to give, naturally for all $A, B: {\sf Set}$, a map of the form
\[ d_{A,B}: \big((p_1 \tri p_2) \otimes (q_1 \tri q_2)\big) (A \times B) \to (p_1 \otimes q_1) \tri (p_2 \otimes q_2) (A \times B) \]
And by the universal property of the Day convolution product, in order to give $d_{A,B}$, it is enough to provide a map of the form
\[ \phi:  (p_1 \tri p_2)(A) \times (q_1 \tri q_2)(B) \to (p_1 \otimes q_1) \tri (p_2 \otimes q_2) (A \times B) \]
which can be constructed as follows:
\begin{align*}
(p_1 \tri p_2)(A) \times (q_1 \tri q_2)(B)  
&= (p_1 \tri (p_2 (A)))\times (q_1 \tri (q_2(B))) \\  
&\to^{\sf Day} (p_1 \otimes q_1) (p_2(A) \times q_2(B)) \\ 
&\to^{\id\tri {\sf Day}} (p_1 \otimes q_1)\big((p_2 \otimes q_2)(A \times B)\big) \\ 
&= (p_1 \otimes q_1) \tri (p_2 \otimes q_2) (A \times B) 
\end{align*}
Then, one can define the duoidal map as the unique map given by the universal property of the Day convolution product as in \eqref{eq.dayUniversalProp}.
\end{proof}

It has been brought to our attention by one of our reviewers that the normal duoidal structure of $\Poly$ given by Dirichlet and functor composition is likely to be closely related\footnote{It seems likely that the normal duoidal structure on $\Poly$ described in Lemma 3.4 may be seen to arise from more general two-dimensional structure. For comparison, the category of species is also equipped with a normal duoidal structure \cite{GaF16}, which was shown in \cite{GGV24} to arise from a normal oplax monoidal structure on the bicategory of coloured symmetric sequences. Abstractly, the bicategory of coloured symmetric sequences is the Kleisli bicategory for a pseudo distributive law between the symmetric monoidal category 2-monad and the presheaf construction relative pseudo monad \cite{FGHW18}. The category of species arises as the presheaf construction on the free symmetric monoidal category on the terminal category 1, and thus as the endo-hom of the bicategory of coloured symmetric sequences on 1. Similarly, the category of polynomial functors arises as the free cocartesian category on the free cartesian category on 1. It is therefore natural to expect that the normal duoidal structure on $\Poly$ arises from normal oplax monoidal structure on the Kleisli bicategory for a pseudo distributive law between the cartesian category 2-monad and the cocartesian category 2-monad.} to the normal duoidal structure of generalized species given by arithmetic and substitution products \cite{MaM08, GaF16, GGV24}. It might be an interesting future direction to explore this connection further.

\begin{corollary}
\label[corollary]{cor: Poly is isomix}
The category $({\sf Poly}, \otimes, \tri, \yon)$ is an isomix non-planar linearly distributive category. 
\end{corollary}
\begin{proof}
By \cref{lemma.poly_normal_duoidal,lemma.normalDuo_Iso}, the category $\Poly$ is isomix LDC. By \cref{lemma.planar}, it is non-planar because it is $\ox$-symmetric.
\end{proof}

Now that we know that $\Poly$ is an isomix LDC, let us examine its mix and $\indep$ maps. In $\Poly$, the units of the $\ox$- and the $\tri$-products coincide: $\bot = \top = \yon$. Hence, the mix map $\m$ can be taken to be the identity map. For any two polynomials $p$ and $q$, the $\indep$ map $p \ox q \to^{\indep} p \tri q$ is defined as follows:
\begin{align}
\label{eq.indep}
\sum_{P:p(1)} \sum_{Q:q(1)} \prod_{d:p[P]} \prod_{b: q[Q]} \yon &\to^{\indep} \nonumber 
\sum_{P:p(1)} \prod_{d:p[P]} \sum_{Q:q(1)} \prod_{b: q[Q]} \yon \\ 
\indep_1: (P,Q) & \mapsto (P, \textit{const} Q: p[P] \to q(1); d \mapsto Q) \nonumber \\
\indep_{P,Q}^\sharp:  (d,b) &\mapsto (d,b): p[P] \times q[Q]
\end{align}
The fact that, under this map, $Q$ arrives in the codomain as a constant, i.e.\ that it is \emph{independent} of $d$, is the reason for the name $\indep$. 
In passing, we note that for each $P:p(1)$ and $Q:q(1)$, the function $\indep_{P,Q}^\sharp$ is a bijection, and hence we record the following.
\begin{corollary}\label[corollary]{cor.indep_cartesian}
For any $p,q:\Poly$, the map $\indep_{p,q}$ is cartesian.
\end{corollary}

\section{Linear duals}
\label{Sec: Duals}

In this section, we introduce biclosed LDCs and characterize the dual objects in these categories. The category $\Poly$ is a biclosed LDC. We apply the general results on biclosed LDCs to identify the dual objects in $\Poly$.

\subsection{Linear duals in LDCs}

Let us recall the definition of dual objects in LDCs.

\begin{definition}[{\cite[Def. 2.8.]{Sri21}}]\label[definition]{def.duals}
     Suppose $\X$ is a LDC and $a,b : \X$ are objects. Then $b$ is {\bf left dual} to $a$---and $a$ is {\bf right dual} to $b$---written as \[ (\eta, \epsilon): b \dual a \] if there exists a unit map $\eta: \top \to a \tri b$, and a counit map $\epsilon: b \otimes a \to \bot$ such that the following diagrams commute:
    \begin{equation}\label{eqn.linduals}
	\textbf{\em [dual.1]}
    \begin{tikzcd}
        a \ar[r, "u_\ox^{-1}"] \ar[d, equals]
        &[-5pt] \top \ox a  \ar[r,"\eta \ox \id"]
        & (a \tri b) \ox a \ar[d, "\partial^R"]
        \\ 
        a \ar[r, <-,  "u_\tri"']
        & a \tri \bot \ar[r, <-,  "1 \tri \epsilon"']
        & a \tri (b \ox a)
    \end{tikzcd}
	\textbf{\em [dual.2]}
    \begin{tikzcd}
        b \ar[r, "u_\ox^{-1}"] \ar[d, equals]
        &[-5pt] b \ox \top  \ar[r,"\id \ox \eta "]
        & b \ox (a \tri b) \ar[d, "\partial^L"]
        \\ 
        b \ar[r, <-,  "u_\tri"']
        & a \tri \bot \ar[r, <-,  "\epsilon \tri 1"']
        & (b \ox a) \tri b
    \end{tikzcd}
    \end{equation}
\end{definition}

Note that the above coherence conditions involve mainly $\eta$, $\epsilon$ and the linear distributivity maps. For easy understanding of the coherences, the $\eta$ and the $\epsilon$ can be drawn as a cap and a cup respectively (by hiding the units), even though they refer to different monoidal products. The two coherence diagrams {\bf [dual.1]} and {\bf[dual.2]} can then be drawn as snake diagrams, as shown below.%
\footnote{All string diagrams in this article are to be read top-to-bottom, i.e.\ in the direction of gravity.}
\[\begin{tikzpicture}
	\begin{pgfonlayer}{nodelayer}
		\node [style=none] (0) at (-3.25, 3.5) {};
		\node [style=none] (1) at (-1.25, 3.5) {};
		\node [style=none] (2) at (-7.25, 4.5) {};
		\node [style=none] (3) at (-5.25, 4.5) {};
		\node [style=none] (4) at (-7.25, 4.75) {$b$};
		\node [style=none] (5) at (-5.25, 4.75) {$a$};
		\node [style=none] (6) at (-1.25, 3.25) {$a$};
		\node [style=none] (7) at (-3.25, 3.25) {$b$};
		\node [style=none] (8) at (-6.5, 2.75) {$\epsilon: b \otimes a \rightarrow \bot$};
		\node [style=none] (9) at (-2.25, 2.75) {$\eta: \top \rightarrow A \tri B$};
		\node [style=none] (10) at (0, 5.75) {};
		\node [style=none] (11) at (0, 2) {};
		\node [style=none] (12) at (1.25, 2.75) {};
		\node [style=none] (13) at (1.25, 4.25) {};
		\node [style=none] (14) at (2, 4.25) {};
		\node [style=none] (15) at (2, 3.5) {};
		\node [style=none] (16) at (2.75, 3.5) {};
		\node [style=none] (17) at (2.75, 5) {};
		\node [style=none] (18) at (3.25, 4) {$=$};
		\node [style=none] (19) at (7.5, 2.75) {};
		\node [style=none] (20) at (7.5, 4.25) {};
		\node [style=none] (21) at (6.75, 4.25) {};
		\node [style=none] (22) at (6.75, 3.5) {};
		\node [style=none] (23) at (6, 3.5) {};
		\node [style=none] (24) at (6, 5) {};
		\node [style=none] (25) at (3.75, 5) {};
		\node [style=none] (26) at (3.75, 2.75) {};
		\node [style=none] (27) at (9, 5) {};
		\node [style=none] (28) at (9, 2.75) {};
		\node [style=none] (29) at (8.25, 4) {$=$};
		\node [style=none] (32) at (2.75, 5) {};
		\node [style=none] (35) at (3.25, 2) {\text{[dual.1]}};
		\node [style=none] (36) at (8.25, 2) {\text{[dual.2]}};
		\node [style=none] (38) at (2.75, 5.25) {$a$};
		\node [style=none] (39) at (1.25, 2.5) {$a$};
		\node [style=none] (40) at (2.25, 4) {$b$};
		\node [style=none] (41) at (6, 5.25) {$b$};
		\node [style=none] (42) at (7.5, 2.5) {$b$};
		\node [style=none] (43) at (6.5, 4) {$a$};
	\end{pgfonlayer}
	\begin{pgfonlayer}{edgelayer}
		\draw [bend left=90, looseness=2.25] (0.center) to (1.center);
		\draw [bend right=90, looseness=2.25] (2.center) to (3.center);
		\draw [dotted] (10.center) to (11.center);
		\draw (12.center) to (13.center);
		\draw [bend left=90, looseness=1.25] (13.center) to (14.center);
		\draw (14.center) to (15.center);
		\draw [bend right=90, looseness=1.25] (15.center) to (16.center);
		\draw (16.center) to (17.center);
		\draw (19.center) to (20.center);
		\draw [bend right=90, looseness=1.25] (20.center) to (21.center);
		\draw (21.center) to (22.center);
		\draw [bend left=90, looseness=1.25] (22.center) to (23.center);
		\draw (23.center) to (24.center);
		\draw (25.center) to (26.center);
		\draw (27.center) to (28.center);
	\end{pgfonlayer}
\end{tikzpicture}
\]

%
%
We record the following definition and lemma for completeness, but they will not be used again in the paper.
\begin{definition}[{\cite[Definition 2.10]{Sri21}}]
			A morphism of linear duals, $(f,g): (\eta,  \epsilon) \to (\eta', \epsilon')$, is given by a pair of maps
			\[ \xymatrix{ b \ar[d]_f \ar@{-||}[r]^{(\eta,  \epsilon)}  &  a \\ b' \ar@{-||}[r]_{(\eta', \epsilon')} &  a' \ar[u]_{g}} \] 
		such that the following equations hold:  \[ (i) ~~~  \begin{tikzpicture}
			\begin{pgfonlayer}{nodelayer}
				\node [style=none] (0) at (-1, 2) {};
				\node [style=none] (1) at (0.5, 3) {};
				\node [style=none] (2) at (-1, 3) {};
				\node [style=none] (3) at (-0.25, 4) {$\eta'$};
				\node [style=none] (4) at (0.5, 2) {};
				\node [style=none] (5) at (-1.25, 2.25) {$b'$};
				\node [style=none] (6) at (1, 3.5) {$a'$};
				\node [style=circle, scale=1.5] (7) at (0.5, 2.75) {};
				\node [style=none] (8) at (0.5, 2.75) {$g$};
				\node [style=none] (9) at (1, 2.25) {$a$};
			\end{pgfonlayer}
			\begin{pgfonlayer}{edgelayer}
				\draw (4.center) to (7);
				\draw (1.center) to (7);
				\draw (2.center) to (0.center);
				\draw [bend left=90, looseness=1.75] (2.center) to (1.center);
			\end{pgfonlayer}
		\end{tikzpicture}
		= \begin{tikzpicture}
			\begin{pgfonlayer}{nodelayer}
				\node [style=none] (0) at (0.5, 2) {};
				\node [style=none] (1) at (-1, 3) {};
				\node [style=none] (2) at (0.5, 3) {};
				\node [style=none] (3) at (-0.25, 4) {$\eta$};
				\node [style=none] (4) at (-1, 2) {};
				\node [style=none] (5) at (0.75, 2.25) {$a$};
				\node [style=none] (6) at (-1.5, 3.5) {$b$};
				\node [style=circle, scale=1.5] (7) at (-1, 2.75) {};
				\node [style=none] (8) at (-1, 2.75) {$f$};
				\node [style=none] (9) at (-1.5, 2.25) {$b'$};
			\end{pgfonlayer}
			\begin{pgfonlayer}{edgelayer}
				\draw (4.center) to (7);
				\draw (1.center) to (7);
				\draw (2.center) to (0.center);
				\draw [bend right=90, looseness=1.75] (2.center) to (1.center);
			\end{pgfonlayer}
		\end{tikzpicture}
		  ~~~~~~(ii)~~~ \begin{tikzpicture}
			\begin{pgfonlayer}{nodelayer}
				\node [style=none] (0) at (0.5, 3.75) {};
				\node [style=none] (1) at (-1, 2.75) {};
				\node [style=none] (2) at (0.5, 2.75) {};
				\node [style=none] (3) at (-0.25, 1.75) {$\epsilon'$};
				\node [style=none] (4) at (-1, 3.75) {};
				\node [style=none] (5) at (0.75, 3.5) {$a'$};
				\node [style=none] (6) at (-1.5, 2.25) {$b'$};
				\node [style=circle, scale=1.5] (7) at (-1, 3) {};
				\node [style=none] (8) at (-1, 3) {$f$};
				\node [style=none] (9) at (-1.5, 3.5) {$b$};
			\end{pgfonlayer}
			\begin{pgfonlayer}{edgelayer}
				\draw (4.center) to (7);
				\draw (1.center) to (7);
				\draw (2.center) to (0.center);
				\draw [bend right=90, looseness=1.75] (1.center) to (2.center);
			\end{pgfonlayer}
		\end{tikzpicture}
		 =  \begin{tikzpicture}
			\begin{pgfonlayer}{nodelayer}
				\node [style=none] (0) at (-1, 3.75) {};
				\node [style=none] (1) at (0.5, 2.75) {};
				\node [style=none] (2) at (-1, 2.75) {};
				\node [style=none] (3) at (-0.25, 1.75) {$\epsilon$};
				\node [style=none] (4) at (0.5, 3.75) {};
				\node [style=none] (5) at (-1.25, 3.5) {$b$};
				\node [style=none] (6) at (1, 2.25) {$a$};
				\node [style=circle, scale=1.5] (7) at (0.5, 3) {};
				\node [style=none] (8) at (0.5, 3) {$g$};
				\node [style=none] (9) at (1, 3.5) {$a'$};
			\end{pgfonlayer}
			\begin{pgfonlayer}{edgelayer}
				\draw (4.center) to (7);
				\draw (1.center) to (7);
				\draw (2.center) to (0.center);
				\draw [bend right=90, looseness=1.75] (2.center) to (1.center);
			\end{pgfonlayer}
		\end{tikzpicture} \]
		\end{definition}
		
		Notice that a morphism of duals is determined completely by either of the pair of maps ($f$ 
		completely determines $g$, and vice versa), and are referred to as Australian mates \cite{CKS00}. 
        \[ g := \begin{tikzpicture}
			\begin{pgfonlayer}{nodelayer}
				\node [style=circle, scale=2] (0) at (0, 2) {};
				\node [style=none] (1) at (0, 2) {$f$};
				\node [style=none] (2) at (0, 2.5) {};
				\node [style=none] (3) at (1, 2.5) {};
				\node [style=none] (4) at (1, 0.75) {};
				\node [style=none] (5) at (0, 1.5) {};
				\node [style=none] (6) at (-1, 1.5) {};
				\node [style=none] (7) at (-1, 3.5) {};
				\node [style=none] (8) at (-1.25, 3) {$a'$};
				\node [style=none] (9) at (1.25, 1.25) {$a$};
			\end{pgfonlayer}
			\begin{pgfonlayer}{edgelayer}
				\draw (2.center) to (0);
				\draw (0) to (5.center);
				\draw [bend left=90, looseness=1.75] (5.center) to (6.center);
				\draw (6.center) to (7.center);
				\draw [bend left=90, looseness=2.00] (2.center) to (3.center);
				\draw (3.center) to (4.center);
			\end{pgfonlayer}
		\end{tikzpicture}
		\]
		
		The map $f$ is an isomorphism if and only if $g$ is an isomorphism. Also, notice that 
		if $f$ or $g$ is an isomorphism, equations $(i)$ and $(ii)$ imply one another in the 
		definition of morphism of duals. 
\begin{lemma}
\label[lemma]{Lemma: eta-epsilon}
In a mix LDC, we have the following equations: 
\[ \begin{tikzpicture}
	\begin{pgfonlayer}{nodelayer}
		\node [style=none] (0) at (-7, 3) {};
		\node [style=none] (1) at (-5, 3) {};
		\node [style=none] (2) at (-5, 5.75) {};
		\node [style=none] (3) at (-3, 5.75) {};
		\node [style=none] (4) at (-5, 6) {$b$};
		\node [style=none] (5) at (-3, 6) {$a$};
		\node [style=none] (6) at (-5, 2.75) {$b$};
		\node [style=none] (7) at (-7, 2.75) {$a$};
	\end{pgfonlayer}
	\begin{pgfonlayer}{edgelayer}
		\draw [bend left=90, looseness=2.50] (0.center) to (1.center);
		\draw [bend right=90, looseness=2.75] (2.center) to (3.center);
	\end{pgfonlayer}
\end{tikzpicture}
 \stackrel{(i)}{=}  \begin{tikzpicture}
	\begin{pgfonlayer}{nodelayer}
		\node [style=none] (0) at (-3.25, 2.75) {};
		\node [style=none] (1) at (-1.25, 2.75) {};
		\node [style=none] (2) at (-3.25, 5.75) {};
		\node [style=none] (3) at (-1.25, 5.75) {};
		\node [style=none] (4) at (-3.25, 6) {$b$};
		\node [style=none] (5) at (-1.25, 6) {$a$};
		\node [style=none] (6) at (-1.25, 2.5) {$b$};
		\node [style=none] (7) at (-3.25, 2.5) {$a$};
	\end{pgfonlayer}
	\begin{pgfonlayer}{edgelayer}
		\draw [bend left=90, looseness=2.25] (0.center) to (1.center);
		\draw [bend right=90, looseness=2.25] (2.center) to (3.center);
	\end{pgfonlayer}
\end{tikzpicture}
 \stackrel{(ii)}{=} \begin{tikzpicture}
	\begin{pgfonlayer}{nodelayer}
		\node [style=none] (0) at (-3.25, 2.75) {};
		\node [style=none] (1) at (-1.25, 2.75) {};
		\node [style=none] (2) at (-5, 5.75) {};
		\node [style=none] (3) at (-3, 5.75) {};
		\node [style=none] (4) at (-5, 6) {$b$};
		\node [style=none] (5) at (-3, 6) {$a$};
		\node [style=none] (6) at (-1.25, 2.5) {$b$};
		\node [style=none] (7) at (-3.25, 2.5) {$a$};
	\end{pgfonlayer}
	\begin{pgfonlayer}{edgelayer}
		\draw [bend left=90, looseness=2.50] (0.center) to (1.center);
		\draw [bend right=90, looseness=2.75] (2.center) to (3.center);
	\end{pgfonlayer}
\end{tikzpicture}
\]
Equivalently, the following squares commute: 
\[ (i) ~~~ \begin{tikzcd} 
    b \ox a \ar[r, "\epsilon"] \ar[d, "u_\ox^{-1}"'] 
    & \bot \ar[r, "\m"]
    & \top \ar[r, "\eta"]
    & a \tri b \ar[dd, equals]
    \\
    \top \ox (b \ox a) \ar[d, "\eta \ox \id"']
    & 
    &
    &
    \\
    (a \tri b) \ox (b \ox a) \ar[r, "\id \ox \epsilon"']
    & (a \tri b) \ox \bot \ar[r, "\id \ox \m"']
    & (a \tri b) \ox \top \ar[r, "u_\ox"']
    & a \tri b
\end{tikzcd} \]
\[ (ii) ~~~ \begin{tikzcd}    
    b \ox a \ar[r, "\epsilon"] \ar[d, "u_\ox^{-1}"'] 
    & \bot \ar[r, "\m"]
    & \top \ar[r, "\eta"]
    & a \tri b \ar[dd, equals]
    \\
    (b \ox a) \ox \top \ar[d, "\id \ox \eta"']
    & 
    &
    &
    \\
    (b \ox a) \ox (a \tri b) \ar[r, "\epsilon \ox \id"']
    & \bot \ox (a \tri b) \ar[r,  "\m \ox \id"']
    & \top \ox (a \tri b)  \ar[r, "u_\ox"']
    & a \tri b
\end{tikzcd} 
\]
\end{lemma}
\begin{proof}
We prove below that $(i)$ commutes. The proof that $(ii)$ commutes follows similarly. 
\[
\begin{tikzcd}[column sep=large]
    b \ox a \ar[r, "\epsilon"] \ar[d, "u_\ox^{-1}"'] \ar[dr, phantom, "\text{=nat.}"]
  & \bot \ar[r, "\m"] \ar[d, "u_\ox^{-1}"] \ar[dr, phantom, "\text{=nat.}"]
  & \top  \ar[r, "\eta"] \ar[d, "u_\ox^{-1}"] \ar[dr, phantom, "\text{=inv.}"]
  &  a \tri b  \ar[dd, bend left, equals]
  \\ 
     \top \ox (b \ox a) \ar[r, "\id \ox \epsilon"] \ar[d, "\eta \ox \id"'] \ar[dr, phantom, "=\ox\text{-bifunctor}"]
  & \top \ox \bot \ar[r, "\id \ox \m"]  \ar[d, "\eta \ox \id"] \ar[dr, phantom, "=\ox\text{-bifunctor}"]
  & \top \ox \top  \ar[r, "u_\ox"] \ar[d, "\eta \ox \id"] \ar[dr, phantom, "\text{=nat.}"]
  & \top \ar[d, "\eta"]  
  \\ 
    (a \tri b) \ox (b \ox a) \ar[r, "\id \ox \epsilon"'] 
  & (a \tri b) \ox \bot \ar[r, "\id \ox \m"']  
  & (a \tri b) \ox \top  \ar[r, "u_\ox"'] 
  & (a \tri b) 
\end{tikzcd}
\]
\end{proof}

\subsection{Duals in biclosed LDCs}
\label{Sec: Biclosed}

In this section, we discuss what happens when the parallel structure ($\otimes$) has a left closure and the sequential structure ($\tri$) has a right coclosure. These are related to dual objects in LDCs. 

\begin{definition} We define a linearly distributive category $(\X, \ox, \top, \tri, \bot)$ to be:
\begin{itemize}[nosep]
\item {\bf $\ox$-closed} if for all $a:\X$, the functor $ - \otimes a : \X \to \X$  has a right adjoint, denoted $ a \lollipop ~ -$,
\item {\bf $\tri$-coclosed} if for all $a: \X$, the functor $ - \tri a: \X \to \X$ has a left adjoint, denoted $ a \slash-$,
\item {\bf biclosed} if $\X$ is both $\ox$-closed and $\tri$-coclosed.
\end{itemize}
\end{definition}

Let $a$ be any object in a $\ox$-closed LDC. Then, by the adjunction, for all $b$, and $c$, we have an isomorphism: 
\begin{equation} 
\label{eqn: closure} 
\X(a \ox b, c) \cong \X(b, a \lollipop c)  
\end{equation}
We refer to the counit of the $a \otimes - \dashv a \lollipop -$ adjunction as the \emph{eval} map: 
\begin{equation}
\label{eqn:eval}
\eval: a \ox (a \lollipop b) \to b
\end{equation}

Let $a$ be any object in a $\tri$-coclosed LDC. Then, by the adjunction, for all $b$, and $c$, we have an isomorphism: 
\begin{equation} 
\label{eqn: coclosure} 
\X \left(a \slash c, b \right) \cong \X(a, b \tri c)  \end{equation} 
We refer to the counit of the $- \tri b \dashv - / b$ adjunction as the \emph{coeval} map: 
\begin{equation}
\label{eqn:coeval}
\coeval: a \to (a \slash b) \tri b   
\end{equation}

\begin{example}[\cite{Spi22, NiS23}] The category $\Poly$ is a biclosed LDC where, for any two polynomials $p, q$:
\begin{align}
   p \lollipop q := [p,q] & =  \prod_{P : p(1)} \sum_{Q : q(1)} \prod_{d : q[Q]} \sum_{d' : p[P]} \yon 
   \label{eqn: Poly-close} 
   \\  
   p \slash q := \coclose{q}{p} & = \sum_{P : p(1)} \yon^{ q \tri p[P]} = \sum_{P : p(1)} \prod_{Q: q(1)}\prod_{d:q[Q]\to p[P]}\yon 
    \label{eqn: Poly-coclose}
\end{align}
Although we will not need it, the coclosure $\coclose{q}{p}$ is the left Kan extension of $p$ along $q$:
\[
\begin{tikzcd}
	\Set\ar[r, "p", ""' name=p]\ar[d, "q"']&
	\Set\\
	\Set\ar[ru, bend right=25pt, "\coclose{q}{p}"', "" name=cc]
	\ar[from=p, to=cc-|p, xshift=-17pt, yshift=-5pt, shorten=2pt, Rightarrow, "\;\;\coeval"]
\end{tikzcd}
\]
\end{example}

In the rest of the paper, without loss of generality, we will assume $\top = \bot$ for isomix LDCs. For an isomix LDC, we will denote the monoidal unit by $\yon$. Since we are mainly concerned with $\Poly$ category, in the rest of the paper, we will use the notation $[a, b]$ for $a \lollipop b$, and $\coclose{b}{a}$ for $a \slash b$ for readability and uniformity.

For any object $a:\X$ in a biclosed isomix LDC, there are two sorts of ``double dual'' maps, $\Phi_a$ and $\Psi_a$, each of which uses both the $\otimes$-closure and the $\tri$-coclosure. They are defined as the following composites:%

(For brevity, we suppress identities and the unitors $u_\ox$ and $u_\tri$ in \eqref{eqn: phip} and \eqref{eqn: psip}. We follow this convention in the rest of the document when writing proofs and composites. For full version of the equations, see Appendix A.)

\begin{equation}
\label{eqn: phip}
\begin{tikzcd}
  a \ar[r, "\coeval"] \ar[rrr, bend right = 15pt, "\Phi_a"'] &[10pt]  \left(\coclose{[a, \yon]}{\yon} \tri [a,\yon] \right) \ox a \ar[r, "\partial^R"] &   \coclose{[a, \yon]}{\yon} \tri (([a,\yon]) \ox a) \ar[r, "\eval"] & \coclose{[a,\yon]}{\yon}
\end{tikzcd}
\end{equation}

\begin{equation}
\label{eqn: psip}
\begin{tikzcd}
  \left[ \coclose{a}{\yon}, \yon \right] \ar[r, "\coeval"] \ar[rrr, bend right = 15pt, "\Psi_a"']
& \left[ \coclose{a}{\yon}, \yon \right] \ox \left( \coclose{a}{\yon} \tri a \right) \ar[r, "\partial^L"] 
& \left( \left[ \coclose{a}{\yon}, \yon \right] \ox \coclose{a}{\yon} \right) 
 \tri a \ar[r, "\eval"] 
 & a
\end{tikzcd}
\end{equation}

\begin{theorem}
\label[theorem]{Thm: retract}
    For an object $a$ in a biclosed isomix LDC, the following are equivalent: 

    \begin{enumerate}[(i), nosep]
    \item $\Phi_a$ has a retraction, 
    \[ a \to^{\Phi_a} \coclose{[a,\yon]}{\yon} 
    \to^{\chi_a} a \]
    and the following composite is the identity on $[a,\yon]$, 
    \begin{multline*}
     [a,\yon] \to^{\coeval} [a,\yon] \otimes \left(\coclose{[a,\yon]}{\yon} \tri [a,\yon] \right) \to^ {\partial^R} 
    \left( [a,\yon] \otimes \coclose{[a,\yon]}{\yon} \right) \tri [a,\yon] \\
    \to^{\chi_a} 
    ([a,\yon] \ox a) \tri [a,\yon]  
    \to^{\eval}   [a,\yon] 
    \end{multline*}
    \item 
    $[a,\yon] \dual a$, with $\epsilon= (\eval: [a, \yon] \ox a \to \yon)$.
    \end{enumerate}
\end{theorem}
\begin{proof}
\begin{description}
    
    \item[{\it (i)} $\Rightarrow$ {\it (ii)}:] Given that $(i)$ holds, we must show that $[a, \yon] \dual a$ with $\epsilon=\eval$. Define $\eta$ as follows:
    \begin{align}
    \label{eqn.eta}
        \eta := \left(\yon \to^{\coeval} \coclose{[a,\yon]}{\yon} \tri [a, \yon] \to^{\chi_a} a \tri [a,\yon]\right)
    \end{align}

    The following diagram proves that {\bf[dual.1]} from \cref{def.duals} holds: 
    
    \[ \begin{tikzcd}
      a \ar[r, "\coeval"] \ar[dd, equals] \ar[rr, bend left = 25 pt, "\eta"] \ar[dr, phantom, "\text{=assumption}"]
      & \left( \coclose{[a,\yon]}{\yon} \tri [a,\yon] \right) \ox a \ar[r, "\chi_a"] \ar[d, "\partial^R"] 
      & (a \tri \left[ a,\yon \right]) \ox a \ar[ddl, bend left = 10pt, "\partial^R"] \ar[dl, phantom, "\text{=nat.}"] \\
      & \coclose{ \left[ a,\yon \right]}{a} \tri (\left[ a,\yon \right] \ox a) \ar[d, "\chi_a"]
      & 
      \\
      a \ar[r, <-, "\epsilon = \eval"'] 
      & a \tri (\left[a,\yon \right] \ox a) 
      & 
    \end{tikzcd}\]
    The proof that {\bf[dual.2]} holds is similar. 

    \item[{\it (ii)} $\Rightarrow$ {\it (i)}:] Given that $[a, \yon] \dual a$ with $\epsilon = \eval$, we must prove that {\it (i)} holds. 
    Since the LDC is $\tri$-coclosed, we know that that $ - \tri [a, \yon]$ has a left adjoint as shown in \cref{eqn: coclosure}. Define $\chi_a\colon\coclose{[a,\yon]}{\yon}\to a$ to be the universal map induced by $\eta$:
    \[ \begin{tikzcd}[column sep=large]
         \yon \ar[d, "\coeval"' ] \ar[dr, "\eta"] & \\ 
        \coclose{[a,\yon]}{\yon} \tri [a,\yon] \ar[r, dotted, "\chi_a \tri \id"'] &  a \tri [a,\yon]
    \end{tikzcd} \]

    We first prove that the specified composite is indeed the identity on $[a,\yon]$:
\[
    \begin{tikzcd}
          {[}a,\yon{]} \ar[r, "\coeval"] \ar[d, equals] \ar[dr, phantom, "\text{=univ. prop.}"', description] 
        & {[}a,\yon] \ox \left( \coclose{[a,\yon]}{\yon} \tri [a,\yon] \right) 
        \ar[d, "\chi_a"] \ar[r, "\partial^L"] \ar[dr, phantom, "\text{=nat.}"]
        & \left( {[}a,\yon] \ox \coclose{[a,\yon]}{\yon} \right) \ar[d, "\chi_a"] 
        \\
        {[}a, \yon] \ar[r, "\eta"] \ar[drr, equals, bend right = 10pt]
        & {[}a, \yon] \ox ([a,\yon] \tri a) \ar[r, "\partial^L"] \ar[dr, phantom, "\text{=[dual.2]}"]
        & \left( [a,\yon] \ox a \right) \tri a \ar[d, "\eval"]
        \\ 
        & 
        & {[}a,\yon] 
    \end{tikzcd}
\]
    Next we show that $\Phi_a \then \chi_a$ is a retract, completing the proof:
\[    
    \begin{tikzcd}
        a \ar[r, "\coeval"] \ar[d, equals] \ar[dr, phantom, "\text{=couniv.~prop.}"] \ar[rrr, bend left = 12 pt, shift left=6pt, "\Phi_a"]
      & \left( \coclose{[a,\yon]}{\yon} \tri [a,\yon] \right) \ox a \ar[r, "\partial^L", ""' name = d]  
      \ar[d, "\chi_a"] \ar[dr, phantom, "\text{=nat.}"]
      & \coclose{[a,\yon]}{\yon} \tri ([a,\yon] \ox a) \ar[r, "\eval"] \ar[d, "\chi_a"] \ar[dr, phantom, "\text{=nat.}"]
      & \coclose{[a,\yon]}{\yon} \ar[d, "\chi_a"] 
      \\ 
        a \ar[r, "\eta"'] \ar[rrr, equals, bend right = 15 pt, shift right=5pt, ""' name = i]       
      & (a \tri [a,\yon]) \ox a \ar[r, "\partial^L"']  
      & a \tri ([a,\yon] \ox a) \ar[r, "\eval"'] 
      & a 
    \end{tikzcd}
 \]  
\end{description}    
\end{proof}

\begin{theorem}
\label[theorem]{Thm: section}
    For an object $b$ in a biclosed isomix LDC, the following are equivalent: 

    \begin{enumerate}[(i), nosep]
    \item $\Psi_b$ has a section, 
    \[ b \to^{\Omega_b} \left[ \coclose{b}{\yon}, \yon \right] \to^{\Psi_b} b \]
    and the following composite is the identity on $\coclose{b}{\yon}$, 
    \begin{multline*} 
    	\coclose{b}{\yon} \to^{\coeval}  \left( \coclose{b}{\yon} \tri b \right) \ox \coclose{b}{\yon} 
        \to^{\partial}  \coclose{b}{\yon} \tri  \left( b  \ox \coclose{b}{\yon} \right) \\
        \to^{\Omega_b}  \coclose{b}{\yon} \tri  \left( \left[ \coclose{b}{\yon}, \yon \right]  \ox \coclose{b}{\yon} \right)  
        \to^{\eval} \coclose{b}{\yon}
    \end{multline*}
    \item 
    $b \dual \coclose{b}{\yon}$ with $\eta = \coeval$ maps $\yon \to \coclose{b}{\yon} \tri b$
    \end{enumerate}
\end{theorem}

\begin{proof}
    [Sketch of proof]
    \begin{description}
    	\item[{\it (i)} $\Rightarrow$ {\it (ii)}:] Given that $(i)$ is true, we must show that $b \dual \coclose{b}{\yon}$ with $\eta = \coeval$. Define $\epsilon$ as follows:
    \begin{align}
    \label{eqn.epsilon}
        \epsilon := \left( b \ox \coclose{b}{\yon} \to^{\Omega_b}  \left[ \coclose{b}{\yon}, \yon \right] \ox \coclose{b}{\yon} \to^{\eval} \yon \right)
    \end{align}
        \item[{\it (ii)} $\Rightarrow$ {\it (i)}:] 
        $b \dual \coclose{b}{\yon}$ with $\eta = \coeval$, we must show that $(i)$ holds. Since the LDC is $\ox$-closed, we know that that $ - \ox \coclose{b}{\yon}$ has a right adjoint, see \cref{eqn: closure}. Define $\Omega_b\colon b\to\left[ \coclose{b}{\yon}, \yon \right]$ to be the universal map induced by $\epsilon$:

        \[ \begin{tikzcd}[column sep=large]
            b \ox \coclose{b}{\yon} \ar[r, dotted, "\Omega_b \ox \id"] \ar[dr, "\epsilon"']
          & \left[ \coclose{b}{\yon}, \yon \right] \ox \coclose{b}{\yon} \ar[d, "\eval"]
          \\
          & 
          \yon
        \end{tikzcd} \] 
    \end{description}
   The rest follows similarly to the proof of \cref{Thm: retract}.
\end{proof}

\begin{theorem}
\label[theorem]{Theorem: dual-retract}

    In a biclosed isomix LDC, suppose there are maps $\yon \to^{\eta} a \tri b$ and $b \ox a \to^{\epsilon} \yon$ defining a linear duality $(\eta, \epsilon): b \dual a$. Then adjoint of $\eta$, that is $\eta'\colon\coclose{b}{\yon} \to a$, is a retraction and the adjoint of $\epsilon$, that is $\epsilon'\colon b \to [a, \yon]$, is a section.
\end{theorem}
\begin{proof}~
		Given that $(\eta, \epsilon): b \dual a$, we first prove that $\coclose{b}{\yon} \to^{\eta'} a$ is a retraction. Define the map $a \to^{\varphi} \coclose{b}{\yon}$ as the following composite:
	    \begin{equation}\label{eqn.sectionofeta'} \varphi := \left(a \to^{\coeval } \left( \coclose{b}{\yon} \tri b \right) \ox a \to^{\partial^R} 
         \coclose{b}{\yon} \tri \left( b \ox a \right)
        \to^{\epsilon} \coclose{b}{\yon}
        \right)
        \end{equation}
        The following commuting diagram shows that $\varphi \then \eta' = \id_a$ making $\eta'$ a retraction. 
        \[ 
        { \footnotesize \begin{tikzcd}[column sep=large]
            a 
          	  \ar[dd, equals] 
            	\ar[ddr, phantom, "\text{=Defn. }\varphi"]
            	 \ar[r, "\coeval"] 
            	 \ar[rr, bend left = 30pt, "\eta"]
            & \left( \coclose{b}{\yon} \tri b \right) \ox a \ar[d, "\partial^R"] 
            	\ar[r, "\eta'"] 
           	 \ar[dr, phantom, "\text{=nat.}"]
            & (a \tri b) \ox a \ar[d, "\partial^R"]
            \\  
            &  \coclose{b}{\yon} \tri \left( b  \ox a \right) 
            	\ar[r, "\eta'"'] 
            	\ar[d, "\epsilon"] 
           	 \ar[dr, phantom, "=\tri- \text{bifunctor~ and nat. of } u_\tri"]
            & a \tri (b \ox a) 
            	\ar[d, "\epsilon"]
            \\ 
            a \ar[r, "\varphi"'] 
            	\ar[rr, equals, bend right]
            & \coclose{b}{\yon} 
            	\ar[r, "\eta'"']
            & a 
        \end{tikzcd}} \]
        Indeed, the large left square commutes by the definition of $\varphi$. The outer diagram commutes by {\bf [dual.1]}. Thus, bottom triangle commutes, making $\eta'$ a retraction. 
        
        Next we prove that $b \to^{\epsilon'} [a,\yon]$ is a section. Define the map $[a, \yon] \to^{\psi} b$ to be the following composite:
         \[ \psi := \left([a, \yon] \to^{\eta} [a, \yon] \ox (a \tri b) \to^{\partial^R} ([a,\yon] \ox a) \tri b 
         \to^{\eval} b\right) \]
        The following commuting diagram shows that $\epsilon' \then \psi = \id_b$, making $\epsilon'$ a section. The outer diagram commutes because of {\bf [dual.2]}. The large right square is the definition of the $\psi$ map.
        \[ { \begin{tikzcd}[column sep=large]
             b 
             	\ar[r, "\epsilon'"] \ar[rr, equals, bend left] 
		 \ar[d, "\eta"'] 
		 \ar[dr, phantom, "=\ox\text{-bifunctor and nat. of } u_\ox^{-1}" ]
             & {\left[a, \yon \right]} \ar[r, "\psi"]  
                 \ar[d, "\eta"'] 
		\ar[ddr, phantom, "\psi :="]
             & b 
             	\ar[dd, equals]
             \\   
             b \ox (a \tri b) 
             	\ar[r, "\epsilon'"] 
            	 \ar[d, "\partial^L"'] \ar[dr, phantom, "\text{=nat.}" ]
             & {[a,\yon]} \ox (a \tri b) 
             	\ar[d, "\partial^L"] 
             & 
             \\
             (b \ox a) \tri b 
             	\ar[r, "\epsilon'"] 
		\ar[rr, bend right = 25pt, "\epsilon"']
             & ({[a,\yon]} \ox a) \tri b 
             	\ar[r, "\eval"]
             & b
         \end{tikzcd} } \]
\end{proof}


\subsection{Duals in $\Poly$}
\label{Sec: Duals Poly}

In this section, we characterize the dual objects of $\Poly$ using its left closure and right coclosure properties. We begin with some simple lemmas.

\begin{lemma} 
\label[lemma]{Lem: simple1}
For any polynomial $p$, there is a natural isomorphism $\coclose{p}{\yon} \cong \yon^{p(1)}$
\end{lemma}
\begin{proof}
By \cref{eqn: Poly-coclose}
\[ \coclose{p}{\yon}\cong\coclose{p}{\yon^1} \cong \yon^{p \tri 1} \cong \yon^{p(1)}  \]
\end{proof}

\begin{lemma}
\label[lemma]{Lem: simple2}
For any set $A$, there are natural isomorphisms in $\Poly$:
\begin{enumerate}[(i), nosep]
\item $[A \yon, \yon] \cong \yon^A$
\item $[\yon^A, \yon] \cong A \yon$
\end{enumerate}
\end{lemma}
\begin{proof}
By \cref{eqn: Poly-close},
\begin{align*}
\textit{(i):} \quad [A \yon, \yon] & \cong \prod_{a:A} \sum_1 \prod_1 \sum_1 \yon \cong \prod_{a:A} \yon \cong \yon^A  \\ 
\textit{(ii):} \quad [\yon^A, \yon] &\cong  \prod_1 \sum_1 \prod_1 \sum_{a:A} \yon \cong \sum_{a:A} \yon \cong A\yon
\end{align*}
\end{proof}
Let us examine the coeval map \eqref{eqn:coeval} in $\Poly$,
\[ \sum_{P:p(1)} \yon^{p[P]}=p \to^{\;\;\coeval\;\;} \coclose{q}{p} \tri q  =\sum_{P: p(1)} \yon^{q \tri p[P]} \tri q  \] 
We can write this out using only $\sum$ and $\prod$:
\[
\sum_{P:p(1)}\prod_{d:p[P]}\yon\to\sum\limits_{P: p(1)}  \prod\limits_{Q:q(1)} \prod\limits_{d: q[Q] \to p[P]}\sum\limits_{Q': q(1)}  \prod_{b:q[Q']}\yon
\]
The coeval map is defined in $\Poly$ as follows:
\begin{equation}
\label{eqn: coeval defn}
 \coeval_1: P \mapsto (P, (Q, d) \mapsto Q) 
 \qqand 
 \coeval_P^\sharp: (Q,d,b) \mapsto d(b).
\end{equation}

\begin{theorem}
\label[theorem]{Thm: dual-retract-Poly}
    For any polynomial $p:\Poly$, the following are equivalent: 
    \begin{enumerate}[(i), nosep]
     	\item $p = \yon^A$ for some $A: \Set$,
        	\item The map $\Phi_p: p \to \coclose{[p,\yon]}{\yon}$ from \eqref{eqn: phip} is the identity,
        \item The map $\Phi_p$ has a retraction, 
        \[ p \to^{\Phi_p} \coclose{[p,\yon]}{\yon} 
        \to^{\chi_p} p \]
      \end{enumerate}
      If one (and hence all) of (i), (ii), and (iii) holds, then it is also the case that the following composite is the identity on $[p,\yon]$, 
        \begin{multline}
        \label{eq.PolyComposite}
	         [p,\yon] \to^{\coeval} [p,\yon] \otimes \left(\coclose{[p,\yon]}{\yon} \tri [p,\yon] \right) \to^ {\partial^L} 
	        \left( [p,\yon] \otimes \coclose{[p,\yon]}{\yon} \right) \tri [p,\yon] \\
	        \to^{\chi_p} 
	        ([p,\yon] \ox p) \tri [p,\yon]  
	        \to^{\eval}   [p,\yon] 
        \end{multline}
\end{theorem}
\begin{proof}
	\begin{description}
		\item[{\it (i)}$\Rightarrow$ {\it (ii)}:] Given that $p = \yon^A$ for some $A: \Set$, using \cref{Lem: simple1,Lem: simple2} we have that $\Phi_p$ is the composite of natural maps: 
\[
		\yon^A \to (\yon^A\tri A\yon)\otimes \yon^A \to \yon^A\tri (A\yon\otimes \yon^A) \to \yon^A
\]
and one can check directly that this is the identity.
 We rewrite it as
\[
		\yon^A\to A^A\yon^{A\times A}\to A^A\yon^{A\times A}\to\yon^A.
\]
The first map sends the unique position to $\id_A$ and on directions sends the pair $(a_1,a_2)\mapsto a_2$. The second map is the identity. The third map sends a position $f\colon A\to A$ to the unique position and a direction $a:A$ to $(a,f(a))$. So the composite sends a direction $a$ to $\id_A(a)=a$.

		\item[{\it (ii)} $\Rightarrow$ {\it (iii)} :]  $\Phi_p$ is the identity, hence is also a retract. 
		\item[{\it (iii)} $\Rightarrow$ {\it (i)}:] Given that $\Phi_p \then \chi_p$ is a retract, since functors preserve retracts, we know that $\Phi_p(1) \then \chi_p(1)$ is a retract: 
			\[ \begin{tikzcd} p(1) \ar[r, "\Phi_p(1)"] \ar[rr, bend right = 20pt, equals] & \yon^{[p,\yon] (1)} (1) \ar[r, "\chi_p(1)"]  & p(1) \end{tikzcd} \] 
			However $\yon^{[p,\yon] (1)} (1) = 1$. Hence, we have that $p(1) = 1$. Thus $p = \yon^A$ for some $A:\Set$.
	\end{description}
	
	Finally, assuming that (i) holds, use \eqref{eqn: coeval defn} to prove that the composite $\Phi_p \then \chi_p$ is the identity map.
\end{proof}

Using the above result, we can deduce that in $\Poly$, each linear polynomial is left dual to its corresponding representable, as shown in the following corollary.
\begin{corollary}
\label[corollary]{corr: linear-representable-duality-eta}
    For any set $A$, there is a linear dual $A \yon \dual \yon^A$ in $\Poly$, with $\eta = \coeval$ and $\epsilon = \eval$.
\end{corollary}
\begin{proof} 
Using \cref{Thm: dual-retract-Poly} we have that the map $\Phi_{\yon^A}$ and therefore the map $\chi_{\yon^A}$ is the identity.  
Then, from \cref{Thm: retract}, it follows that $[\yon^A, \yon] \dual \yon^A$ with $\epsilon = \eval$ and $ \eta = \coeval \then \chi_{\yon^A} = \coeval$. 
Finally, from \cref{Lem: simple2}-$(ii)$, we know that $[\yon^A, \yon] = A \yon$. 
\end{proof}

\begin{theorem}
\label[theorem]{Thm: dual-section-Poly}
    For any polynomial $p:\Poly$, the following are equivalent: 
    \begin{enumerate}[(i), nosep]
            \item $q = A \yon$ for some $A: \Set$,
            \item The map $\Psi_q: q \to \left[ \coclose{q}{\yon}, \yon \right] $ in \eqref{eqn: psip}  is the identity,
            \item The map $\Psi_q$ from has a section, 
            \[ q \to^{\Omega_q} \left[ \coclose{q}{\yon}, \yon \right] \to^{\Psi_q} q \]
             If one (and hence all) of $(i)$, $(ii)$, and $(iii)$ holds, it is also the case that the following composite is the identity on $\coclose{q}{\yon}$.
    \end{enumerate}
            \begin{multline}
            \label{eq.PolyComposite_2}
             \coclose{q}{\yon} \to^{\coeval}  \left( \coclose{q}{\yon} \tri q \right) \ox \coclose{q}{\yon} 
                \to^{\partial^L}  \coclose{q}{\yon} \tri  \left( q  \ox \coclose{q}{\yon} \right) 
                \to^{\Omega_q}  \coclose{q}{\yon} \tri  \left( \left[ \coclose{q}{\yon}, \yon \right]  \ox \coclose{q}{\yon} \right)  
                \to^{\eval} \coclose{q}{\yon}
            \end{multline}
\end{theorem}
\begin{proof} The proof for {\it (i)} $\Rightarrow$ {\it (ii)} and  {\it (ii)} $\Rightarrow$ {\it (iii)} are similar to that of \cref{Thm: dual-retract-Poly}. For {\it (iii)} $\Rightarrow$ {\it (i)}, we have by \cref{Lem: simple1,Lem: simple2} that
\[ \left[ \coclose{q}{\yon}, \yon \right] =   \left[ \yon^{q(1)}, \yon \right] = q(1) \yon \] 
Since functors preserve retracts, applying $\Gamma_{(-)}: \Poly \to \Set^{\op}$ from \cref{eqn.gamma} to the retract in $(iii)$, we get the following composite to be a retract in $\Set$:
\[ \begin{tikzcd} 
\Gamma_q \ar[r, "\Gamma_{\Psi_q}" ] 
& \Gamma_{q(1) \yon} \ar[r, "\Gamma_{\Omega_q}"]
&\Gamma_q  
\end{tikzcd} \]
Since $\Gamma_{q(1)\yon}=1$, it follows that $\Gamma_q=1$ and hence that $q[Q]=1$ for all $Q:q(1)$. So $q = q(1) \yon$.
\end{proof}

Next we show that the linear polynomials and the representables are the only objects with duals in $\Poly$:

\begin{theorem}
\label[theorem]{Thm: only Duals in Poly}
    In the category $\Poly$, if $(\eta, \epsilon): q \dual p$, then $p = \yon^A$ and $q = A\yon$ for some $A : {\sf Set}$.  
\end{theorem}

\begin{proof}
    Given that $q \dual p$ is a linear dual, we first show that for $A,B : \Set$, $p = \yon^A$ and $q = B \yon$, and then that $A \cong B$.
    
    We know from \cref{Theorem: dual-retract} the following composite is a retract:
    \begin{align}
    \label{eqn: retract1}
    \begin{tikzcd}[ampersand replacement=\&]
         p \ar[r, "\varphi"] \& \coclose{q}{\yon} \ar[r, "\eta'"] \& p \end{tikzcd} 
    \end{align}
    where $\varphi$ is as in \eqref{eqn.sectionofeta'}.
     By \cref{Lem: simple1}, we have $\coclose{q}{\yon} \cong \yon^{q(1)}$, and hence $\coclose{q}{\yon}(1)=1$. Applying the functor $-(1): \Poly \to \Set$ to the retract in \eqref{eqn: retract1}, we have the following retract:
     \[ p(1) \to^{\varphi_1} 1 \to^{\eta'_1} p(1) \]
     so $p(1)=1$, and hence $p=\yon^A$ for some $A:\Set.$

    From \cref{Theorem: dual-retract}, we have the other retract:
    \begin{align} 
    \label{eqn: retract2}
     q \to^{\epsilon'} [p,\yon] \to^{\psi} q 
    \end{align}
    From the previous step we know that $p = \yon^A$. Substituting, we get $[p,\yon] = [\yon^A, \yon] = A\yon$. Hence, we have the following retracts:
    \[q \to^{\epsilon'} A\yon \to^{\psi} q \]
    Again using $\Gamma$ (as in the proof of \cref{Thm: dual-section-Poly}), we find that $q=B\yon$ for some $B:\Set$.
    
    So we get $\coclose{q}{\yon} = \coclose{B\yon}{\yon} = \yon^B$. Substituting in \eqref{eqn: retract1}, we get the retracts:  
    \begin{align*}
     \infer{ A \to^{\eta'^\sharp} B \to^{\varphi^\sharp} A}{ \yon^A \to^{\varphi}  \yon^B \to^{\eta'} \yon^A }
    \end{align*}    
    Similarly, using the fact that $p=\yon^A$ for some $A: \Set$, we have that $[p, \yon] = [\yon^A, \yon] = A \yon$. Substituting in  \eqref{eqn: retract2}, we get the following retracts: 
    \begin{align*}
    	\infer{ B \to^{\epsilon'_1} A \to^{\psi_1} B }{ B \yon \to^{\epsilon'} A\yon \to^{\psi} B \yon }
    \end{align*}
    These two retracts together form the isomorphism $A \cong B$.
  %
\end{proof}

\begin{lemma}
\label[lemma]{lemma.cyclicdual}
In $\Poly$, if $p\dual q$ and $q \dual p$ then $p = q = \yon$.
\end{lemma}
\begin{proof}
This follows directly from \cref{Thm: only Duals in Poly}.
\end{proof}

%

\section{The core of {\sf Poly}}
\label{Sec:Core}

Central to mix LDCs is the notion of \emph{core}. An important way in which compact LDCs (for all objects $a,b$, an isomorphism $a \ox b \to^{\indep}_{\cong} a \tri b$) arise is as the {\it core} of a mix category. The core of a mix LDC is defined as follows.

\begin{definition}[\cite{BCS00}] An object $u$ is in the {\bf core} of a mix category if and only if 
	the following natural transformations are isomorphisms: 
    \[ u \ox - \to^{\indep_{u,-}} u \tri - ~~~~\mbox{and}~~~~ 
       - \ox  u \to^{\indep_{-,u}} - \tri u \] 
\end{definition}
To the best of our knowledge, the notion of core has been studied only within symmetric LDCs. In this article, motivated by $\Poly$, we consider a more general notion of core, suitable in (possibly non-symmetric) LDCs. 

\begin{definition}
For a {\bf mix ($\ox$-symmetric) LDC} $(\X, \otimes, \top, \tri, \bot)$, 

\begin{itemize}
\item its {\bf left core} is defined to be the full sub-category of objects $a: \X$ such that, for all $x:\X$, the map
\[ {\sf [core_l]} ~~~~ a \otimes x \to^{\indep_{a,x}}_{\cong} a \tri x \]
 is an isomorphism. 
\item its {\bf right core} of $\X$ is the full sub-category of objects $a: \X$ such that, for all $x:\X$, the map
\[ {\sf [core_r]} ~~~~ x \otimes a \to^{\indep_{x,a}}_{\cong} x \tri a \]
is an isomorphism.
\end{itemize}
\end{definition}

One may have empty left and right cores in a mix LDC. However, in an isomix category, the monoidal unit belongs to both the left and the right core. 

\begin{lemma}
If $\X$ is an isomix LDC with unit $\yon$, then $\yon \in \lCore(\X)$ and $\yon \in \rCore(\X)$.
\end{lemma}

In an isomix LDC, both the left and the right cores are compact LDCs: that is, for all objects $a, b :\lCore$, the map $a \ox b \to^{\indep}_{\cong} a \tri b$ is an isomorphism, and same for $\rCore$. Hence, by Corollary \cite[Corollary 2.18]{Sri21} both the left and right cores are linearly equivalent to monoidal categories.

In a symmetric mix LDC $\X$, the left and the right cores coincide. The left and right cores of any isomix LDC may be related to each other as follows.

\begin{definition}
    A mix LDC $(\X, \otimes, \top, \tri, \bot)$ is said to have {\bf opposing cores} if there exists an isomorphism,
    \[ \star: \rCore(\X)^{\op} \to^{\cong} \lCore(\X) \]   
\end{definition}
\begin{example}
Compact closed categories have opposing cores which coincide (the entire category is the core) and the $\star$ is the same as the contravariant involution.
\end{example}
\begin{example}
We will show that $\Poly$ has opposing cores in \cref{corr.opposing}. 
\end{example}
The following theorem generalizes the result \cite{BCS00} that for any linear dual in a symmetric mix LDC, if the left dual is in the core then the right dual is also in the core, and vice versa. We prove an analogous result in non-symmetric isomix LDCs.
\begin{lemma} 
   Let $(\eta, \epsilon): b \dual a$ be a linear dual  in an isomix LDC $\X$. Then $b :\rCore(\X)$ iff $a :\lCore(\X)$.
\end{lemma}
\begin{proof}~
We have $\eta\colon\yon\to a\tri b$ and $\epsilon\colon (b\otimes a)\to\yon$ satisfying the diagrams in \eqref{eqn.linduals}. 
    \begin{description}
        \item[($\Leftarrow$):] Suppose $b :\rCore(\X)$. We want to prove that $a :\lCore(\X)$, i.e.\ that the map $\indep_{a,x} : a \ox x \to a \tri x$ is an isomorphism for all $x: \X$. 
        
        Define a tentative inverse map $a \tri x \to a \ox x$ as follows:
        \begin{multline*}
            a \tri x  \to^{\eta}  (a \tri b) \ox (a \tri x)  
             \to^{\indep_{a,b}^{-1}} (a \ox b) \ox (a \tri x)  \\
             \to^{a_\ox}_{\cong} a \ox (b \ox (a \tri x)) 
             \to^{\partial^L} a \ox ((b \ox a) \tri x)  
             \to^{\epsilon} a \ox x 
        \end{multline*}
        The second step uses the assumption that $b :\rCore(\X)$. The proof that the tentative inverse map defined as above is indeed the inverse of $\indep_{a,x}$ follows similarly to \cite[Proposition 5]{BCS00} and uses the inverse of the mix map $\m$. 
                
           \item[($\Rightarrow$):] [Sketch of proof] The proof of the converse is similar. 
 Suppose $a :\lCore(\X)$. We want to prove that $b:\rCore(\X)$, that is, the map $\indep_{x,b}: x \ox b \to x \tri b$ is an isomorphism for all $x:\X$.
       
        One checks that a map $x \tri b \to x \ox b$ inverse to $\indep_{x,b}$ can be defined as follows: 
       \begin{multline*}
           x \tri b  \to^{\eta}  (x \tri b) \ox (a \tri b)  
             \to^{\indep_{a,b}^{-1}} (x \tri b) \ox (a \otimes b) \\
             \to^{a_\ox}_{\cong} ((x \tri b) \ox a) \otimes b 
             \to^{\partial^L} (x \tri (b \ox a)) \otimes b  
             \to^{\epsilon} x \ox b 
        \end{multline*}
 	
    \end{description}
\end{proof}

The next result shows that in $\Poly$ only the linear polynomials reside in the left core, and only the representables reside in the right core. 
\begin{lemma} 
\label[lemma]{Lemma: left-right}
In the category {\sf Poly}, the following statements hold for any polynomial $p$:
\begin{enumerate}[(i), nosep]
\item  $p\in\lCore(\Poly)$ iff $p \cong A\yon$ for some $A:\Set$, and

\item $q\in\rCore(\Poly)$ iff $q \cong \yon^A$ for some $A:\Set$.

\end{enumerate}
\end{lemma}

\begin{proof}~\\
    \begin{enumerate}[{\it (i)}, nosep]
        \item  
        \begin{description}
            \item[($\Leftarrow$):] Suppose $p \cong A\yon$ and $r$ is any polynomial. Then, 
            \begin{align*}
                A\yon \ox r \cong Ar \cong A\yon \tri r
            \end{align*}
            Thus, $A\yon \in \lCore(\Poly)$. 
            \item[($\Rightarrow$):] Suppose $p :\lCore(\Poly)$. So, for all $r: \Poly$, we have that $\indep: p \ox r \to p \tri r$ is an isomorphism. 
            Let $r \coloneqq 2$ (any set of cardinality $>1$ will work). We have
                \begin{align*}
                     p \ox r &= p \ox 2 = \sum_{P:p(1)} 2 \\ 
                     p \tri r &= p \tri 2 = \sum_{P :p(1)} 2^{p[P]}
                \end{align*}
and in this case $\indep$ is the $p(1)$-indexed sum of diagonals, i.e.\ the function
\[ \sum_{P:p(1)} 2 \xrightarrow{\sum_{P:p(1)}\Delta} \sum_{P :p(1)} 2^{p[P]}. \]
This is a bijection iff  $p[P]=1$ for all $P:p(1)$. Thus $p=p(1)\yon$ is of the required form.
        \end{description}
        \item 
        \begin{description}
            \item[($\Leftarrow$):] Suppose $q \cong \yon^A$, we must prove that for all $p: \Poly$, the map $p \ox \yon^A \to^{\indep_{p, \yon^A}} p \tri \yon^A$ is an isomorphism. Using \eqref{eq.indep}, the map $\indep_{q,\yon^A}$ is defined as follows:
				\begin{align*}
				p \ox \yon^A = \sum_{P:p(1)} \prod_{d:p[P]} \prod_{a: A} \yon &\to^{\indep_{p,\yon^A}} \nonumber 
				\sum_{P:p(1)} \prod_{d:p[P]} \prod_{a: A} \yon = p \tri \yon^A \\ 
				 \indep_1: (P,!) &\mapsto (P, \textit{const}~!: p[P] \to \{!\}; d \mapsto !)   \\  
				\indep_{P,Q}^\sharp:  (d,a) &\mapsto (d,a): p[P] \times q[Q] 
				\end{align*}
           Thus, the map $\indep_1$ is an isomorphism, and $\indep^\sharp$ is always an isomorphism by \cref{cor.indep_cartesian}. Thus, we have that $q : \rCore(\Poly)$.
            \item[($\Rightarrow$):] Suppose $q :\rCore(\Poly)$, so for all $x :\Poly$, the following map is an isomorphism:
            \[ \indep_{x,q}: x \ox q \to^{\cong} x \tri q \]
            Then taking $x=1$ we have $1\ox q \cong 1\tri q\cong 1$. But from \cref{lemma.tensored_with_1} we also have $1\ox q \cong q(1)$. Hence $q(1)\cong 1$, making $q$ a representable.
            
         \end{description}
         
    \end{enumerate}
\end{proof}

\begin{corollary}
\label[corollary]{corr.opposing}
    The category $\Poly$ has opposing cores. 
\end{corollary}
\begin{proof}
	
	For any $A: \Set$, the object  $\left(\yon^A \right)^\star := A\yon$. By \cref{Lemma: left-right} this defines a mapping on objects (up to isomorphism, but any choice is sufficient).
	
	For any $A,B: \Set$ and a map $\varphi: \yon^B \to \yon^A$, the map $\varphi^\star: A \yon \to B\yon$ is given by $\varphi$, where $(\varphi^\star)_1 := \varphi^\sharp$.
\end{proof}

It also follows from \cref{Lemma: left-right} that the left duals in $\Poly$ are precisely the left core of $\Poly$ and the right duals are the right core.
\begin{corollary}
    In the category ${\sf Poly}$, for any two polynomials $p$ and $q$, we have $q \dual p$ iff $q : \lCore(\Poly)$ and $p : \rCore(\Poly)$. 
\end{corollary}
\begin{proof}
    Suppose that $q\dual p$, then by \cref{Thm: only Duals in Poly}, we have that $q=A \yon$ and $p = \yon^A$ for some set $A$. The rest of the proof is the direct consequence of \cref{Lemma: left-right}. For the converse, apply \cref{corr: linear-representable-duality-eta}.
\end{proof}

\section{Linear monoids, linear comonoids, linear bialgebras in $\Poly$}
\label{Sec.linmon}

When an object in a cyclic linear dual ($b \dual a$ and $a \dual b$) is equipped with a $\ox$-monoid structure, it is called a linear monoid. Similarly, if such an object is equipped with a $\ox$-comonoid structure, it is called a linear comonoid. Linear monoids and linear comonoids in LDCs are analogous to Frobenius algebras in monoidal categories. Linear monoids and linear comonoids were studied in the context of quantum observables in \cite[Chapter 9]{Sri21}.

In $\Poly$, the only cyclic linear dual is $a=b=\yon$ by \cref{lemma.cyclicdual}. However, we can consider the notions of left and right linear monoids and comonoids on duals that are not cyclic. This section, focuses on these structures in the category $\Poly$.

In what follows, we are aiming to define a notion of right and left linear monoids and comonoids: four cases. Each is about a linear dual $b\dual a$ such that both sides carry additional structure. The word left and right will refer to the side that carries the $\otimes$-structure, and in this case the other side will (implicitly) carry a complementary $\tri$-structure. For example, a left linear comonoid is a dual $b\dual a$ such that the left side, $b$ carries the $\otimes$-comonoid structure, and the terminology leaves implicit the fact that $a$ carries a $\tri$-monoid structure.

\subsection{Linear monoids}

Before defining left and right linear monoids, let's consider the standard definition that we are generalizing. In an LDC, a linear monoid consists of linear duals $a \dual b$ and $b \dual a$, with either $a$ or $b$ carrying a $\ox$-monoid structure and satisfying certain coherences \cite[Definition 8.7]{Sri21}. Thus, a linear monoid requires cyclic duals, that is, each object is both the left and the right dual of the other. 

\begin{definition}
\label[definition]{defn.linear_monoid}
In any LDC, a {\bf linear monoid}, $b  \linmonw a$, consists of linear duals \\ $(\eta_L, \epsilon_L):b \dual a$ and $(\eta_R, \epsilon_R): a \dual b$,  and a $\ox$-monoid $(b, \mu, \nu)$ such that the $\tri$-comonoid structures induced on $a$ by the dualities coincide: 
\begin{equation}
\label{eq.linmon}
	(i)~~~	
		 \begin{tikzpicture}
		\begin{pgfonlayer}{nodelayer}
			\node [style=circle] (0) at (-2.75, 0.75) {};
			\node [style=none] (1) at (-3.25, 1.25) {};
			\node [style=none] (2) at (-2.75, 0.5) {};
			\node [style=none] (3) at (-2.25, 1.25) {};
			\node [style=none] (4) at (-1.5, 1.25) {};
			\node [style=none] (5) at (-1.5, -0) {};
			\node [style=none] (6) at (-3.25, 1.5) {};
			\node [style=none] (7) at (-1, 1.5) {};
			\node [style=none] (8) at (-1, -0) {};
			\node [style=none] (9) at (-3.75, 0.5) {};
			\node [style=none] (10) at (-3.75, 2.25) {};
			\node [style=none] (11) at (-4, 2) {$a$};
			\node [style=none] (12) at (-1.75, 0.25) {$a$};
			\node [style=none] (13) at (-0.75, 0.25) {$a$};
			\node [style=none] (14) at (-2.25, 2.5) {$\eta_R$};
			\node [style=none] (15) at (-2, 1.75) {$\eta_R$};
			\node [style=none] (16) at (-3.25, -0.25) {$\epsilon_R$};
		\end{pgfonlayer}
		\begin{pgfonlayer}{edgelayer}
			\draw [in=150, out=-90, looseness=1.00] (1.center) to (0);
			\draw [in=-90, out=30, looseness=1.00] (0) to (3.center);
			\draw (0) to (2.center);
			\draw [bend left=90, looseness=1.50] (2.center) to (9.center);
			\draw (9.center) to (10.center);
			\draw (6.center) to (1.center);
			\draw [bend left=90, looseness=1.25] (6.center) to (7.center);
			\draw (7.center) to (8.center);
			\draw (4.center) to (5.center);
			\draw [bend right=90, looseness=1.25] (4.center) to (3.center);
		\end{pgfonlayer}
	\end{tikzpicture} = \begin{tikzpicture}
		\begin{pgfonlayer}{nodelayer}
			\node [style=circle] (0) at (-2, 0.75) {};
			\node [style=none] (1) at (-1.5, 1.25) {};
			\node [style=none] (2) at (-2, 0.5) {};
			\node [style=none] (3) at (-2.5, 1.25) {};
			\node [style=none] (4) at (-3.25, 1.25) {};
			\node [style=none] (5) at (-3.25, -0) {};
			\node [style=none] (6) at (-1.5, 1.5) {};
			\node [style=none] (7) at (-3.75, 1.5) {};
			\node [style=none] (8) at (-3.75, -0) {};
			\node [style=none] (9) at (-1, 0.5) {};
			\node [style=none] (10) at (-1, 2.25) {};
			\node [style=none] (11) at (-0.75, 2) {$a$};
			\node [style=none] (12) at (-3, 0.25) {$a$};
			\node [style=none] (13) at (-4, 0.25) {$a$};
			\node [style=none] (14) at (-2.5, 2.5) {$\eta_L$};
			\node [style=none] (15) at (-2.75, 1.75) {$\eta_L$};
			\node [style=none] (16) at (-1.5, -0.25) {$\epsilon_L$};
		\end{pgfonlayer}
		\begin{pgfonlayer}{edgelayer}
			\draw [in=30, out=-90, looseness=1.00] (1.center) to (0);
			\draw [in=-90, out=150, looseness=1.00] (0) to (3.center);
			\draw (0) to (2.center);
			\draw [bend right=90, looseness=1.50] (2.center) to (9.center);
			\draw (9.center) to (10.center);
			\draw (6.center) to (1.center);
			\draw [bend right=90, looseness=1.25] (6.center) to (7.center);
			\draw (7.center) to (8.center);
			\draw (4.center) to (5.center);
			\draw [bend left=90, looseness=1.25] (4.center) to (3.center);
		\end{pgfonlayer}
	\end{tikzpicture} 
	~~~~~~~~~~~~ 
	(ii)~~~ 	
	\begin{tikzpicture}
		\begin{pgfonlayer}{nodelayer}
			\node [style=circle] (0) at (-0.75, 1.5) {};
			\node [style=none] (1) at (-0.75, 0.5) {};
			\node [style=none] (2) at (-1.75, 0.5) {};
			\node [style=none] (3) at (-1.75, 2.25) {};
			\node [style=none] (4) at (-2, 2) {$a$};
			\node [style=none] (5) at (-1.25, -0.25) {$\epsilon_R$};
		\end{pgfonlayer}
		\begin{pgfonlayer}{edgelayer}
			\draw (0) to (1.center);
			\draw [bend left=90, looseness=1.50] (1.center) to (2.center);
			\draw (2.center) to (3.center);
		\end{pgfonlayer}
	\end{tikzpicture} = 
	\begin{tikzpicture}
		\begin{pgfonlayer}{nodelayer}
			\node [style=circle] (0) at (-2, 1.5) {};
			\node [style=none] (1) at (-2, 0.5) {};
			\node [style=none] (2) at (-1, 0.5) {};
			\node [style=none] (3) at (-1, 2.25) {};
			\node [style=none] (4) at (-0.75, 2) {$a$};
			\node [style=none] (5) at (-1.5, -0.25) {$\epsilon_L$};
		\end{pgfonlayer}
		\begin{pgfonlayer}{edgelayer}
			\draw (0) to (1.center);
			\draw [bend right=90, looseness=1.50] (1.center) to (2.center);
			\draw (2.center) to (3.center);
		\end{pgfonlayer}
	\end{tikzpicture} 
\end{equation}
\end{definition}

We now fine-grain this notion of linear monoid by considering linear duals which are non-cyclic. We separate the left and right notions as \cref{def.llm,def.rlm} so that we can include diagrams that explain the implicit $\tri$-comonoid structures.

\begin{definition}\label[definition]{def.llm}
	In an LDC, a {\bf left linear monoid} $(\eta_L, \epsilon_L): b \linmonwl a$, consists of a linear dual $(\eta_L, \epsilon_L): b \dual a$ equipped with a $\ox$-monoid structure $(b, \mu: b \ox b \to b, \nu: \top \to b)$ on the left dual.
\end{definition}
The reason for the ``$\tri$'' near the $a$ in our notation for a left linear monoid $b \linmonwl a$ is that the $\ox$-monoid structure on $b$ induces a $\tri$-comonoid structure $(a, \delta, \gamma)$ on $a$, as can be seen from the following diagram:
	\begin{equation*}
		(i)~~~~\delta = \begin{tikzpicture}
			\begin{pgfonlayer}{nodelayer}
				\node [style=none] (12) at (1.75, 0.25) {$a$};
				\node [style=none] (13) at (0.25, 0.25) {$a$};
				\node [style=circle] (17) at (1, 1.25) {};
				\node [style=none] (18) at (1, 2.5) {};
				\node [style=none] (19) at (0.5, 0) {};
				\node [style=none] (20) at (1.5, 0) {};
				\node [style=none] (21) at (1.25, 2.25) {$a$};
			\end{pgfonlayer}
			\begin{pgfonlayer}{edgelayer}
				\draw [in=-150, out=90, looseness=1.25] (19.center) to (17);
				\draw [in=90, out=-30, looseness=1.25] (17) to (20.center);
				\draw (17) to (18.center);
			\end{pgfonlayer}
		\end{tikzpicture} := 		
     \begin{tikzpicture}
		\begin{pgfonlayer}{nodelayer}
			\node [style=circle] (0) at (-2, 0.75) {};
			\node [style=none] (1) at (-1.5, 1.25) {};
			\node [style=none] (2) at (-2, 0.5) {};
			\node [style=none] (3) at (-2.5, 1.25) {};
			\node [style=none] (4) at (-3.25, 1.25) {};
			\node [style=none] (5) at (-3.25, -0) {};
			\node [style=none] (6) at (-1.5, 1.5) {};
			\node [style=none] (7) at (-3.75, 1.5) {};
			\node [style=none] (8) at (-3.75, -0) {};
			\node [style=none] (9) at (-1, 0.5) {};
			\node [style=none] (10) at (-1, 2.25) {};
			\node [style=none] (11) at (-0.75, 2) {$a$};
			\node [style=none] (12) at (-3, 0.25) {$a$};
			\node [style=none] (13) at (-4, 0.25) {$a$};
			\node [style=none] (14) at (-2.5, 2.5) {$\eta_L$};
			\node [style=none] (15) at (-2.75, 1.75) {$\eta_L$};
			\node [style=none] (16) at (-1.5, -0.25) {$\epsilon_L$};
		\end{pgfonlayer}
		\begin{pgfonlayer}{edgelayer}
			\draw [in=30, out=-90, looseness=1.00] (1.center) to (0);
			\draw [in=-90, out=150, looseness=1.00] (0) to (3.center);
			\draw (0) to (2.center);
			\draw [bend right=90, looseness=1.50] (2.center) to (9.center);
			\draw (9.center) to (10.center);
			\draw (6.center) to (1.center);
			\draw [bend right=90, looseness=1.25] (6.center) to (7.center);
			\draw (7.center) to (8.center);
			\draw (4.center) to (5.center);
			\draw [bend left=90, looseness=1.25] (4.center) to (3.center);
		\end{pgfonlayer}
	\end{tikzpicture} 
	~~~~~~~~~~~~ 
	(ii)~~~~\gamma = \begin{tikzpicture}
		\begin{pgfonlayer}{nodelayer}
			\node [style=circle] (17) at (1, 0.25) {};
			\node [style=none] (18) at (1, 2.5) {};
			\node [style=none] (21) at (1.25, 2.25) {$a$};
		\end{pgfonlayer}
		\begin{pgfonlayer}{edgelayer}
			\draw (17) to (18.center);
		\end{pgfonlayer}
	\end{tikzpicture} := 	
	\begin{tikzpicture}
		\begin{pgfonlayer}{nodelayer}
			\node [style=circle] (0) at (-2, 1.5) {};
			\node [style=none] (1) at (-2, 0.5) {};
			\node [style=none] (2) at (-1, 0.5) {};
			\node [style=none] (3) at (-1, 2.25) {};
			\node [style=none] (4) at (-0.75, 2) {$a$};
			\node [style=none] (5) at (-1.5, -0.25) {$\epsilon_L$};
		\end{pgfonlayer}
		\begin{pgfonlayer}{edgelayer}
			\draw (0) to (1.center);
			\draw [bend right=90, looseness=1.50] (1.center) to (2.center);
			\draw (2.center) to (3.center);
		\end{pgfonlayer}
	\end{tikzpicture} 
    \end{equation*}
	
\begin{definition}\label[definition]{def.rlm}
	In an LDC, a {\bf right linear monoid} $(\eta_R, \epsilon_R): b \linmonwr a$ consists of a linear dual  $(\eta_R, \epsilon_R): b \dual a$, equipped with a $\ox$-monoid structure $(a, \mu: a \ox a \to a, \nu: \top \to a)$ on the right dual.
\end{definition}
The $\ox$-monoid structure on $b$ induces a $\tri$-comonoid structure on $(a,\delta, \gamma)$ via the duality as follows:
\begin{equation*}
		(i)~~~\delta = \begin{tikzpicture}
			\begin{pgfonlayer}{nodelayer}
				\node [style=none] (12) at (1.75, 0.25) {$b$};
				\node [style=none] (13) at (0.25, 0.25) {$b$};
				\node [style=circle] (17) at (1, 1.25) {};
				\node [style=none] (18) at (1, 2.5) {};
				\node [style=none] (19) at (0.5, 0) {};
				\node [style=none] (20) at (1.5, 0) {};
				\node [style=none] (21) at (1.25, 2.25) {$b$};
			\end{pgfonlayer}
			\begin{pgfonlayer}{edgelayer}
				\draw [in=-150, out=90, looseness=1.25] (19.center) to (17);
				\draw [in=90, out=-30, looseness=1.25] (17) to (20.center);
				\draw (17) to (18.center);
			\end{pgfonlayer}
		\end{tikzpicture} := 		
		 \begin{tikzpicture}
		\begin{pgfonlayer}{nodelayer}
			\node [style=circle] (0) at (-2.75, 0.75) {};
			\node [style=none] (1) at (-3.25, 1.25) {};
			\node [style=none] (2) at (-2.75, 0.5) {};
			\node [style=none] (3) at (-2.25, 1.25) {};
			\node [style=none] (4) at (-1.5, 1.25) {};
			\node [style=none] (5) at (-1.5, -0) {};
			\node [style=none] (6) at (-3.25, 1.5) {};
			\node [style=none] (7) at (-1, 1.5) {};
			\node [style=none] (8) at (-1, -0) {};
			\node [style=none] (9) at (-3.75, 0.5) {};
			\node [style=none] (10) at (-3.75, 2.25) {};
			\node [style=none] (11) at (-4, 2) {$b$};
			\node [style=none] (12) at (-1.75, 0.25) {$b$};
			\node [style=none] (13) at (-0.75, 0.25) {$b$};
			\node [style=none] (14) at (-2.25, 2.5) {$\eta_R$};
			\node [style=none] (15) at (-2, 1.75) {$\eta_R$};
			\node [style=none] (16) at (-3.25, -0.25) {$\epsilon_R$};
		\end{pgfonlayer}
		\begin{pgfonlayer}{edgelayer}
			\draw [in=150, out=-90, looseness=1.00] (1.center) to (0);
			\draw [in=-90, out=30, looseness=1.00] (0) to (3.center);
			\draw (0) to (2.center);
			\draw [bend left=90, looseness=1.50] (2.center) to (9.center);
			\draw (9.center) to (10.center);
			\draw (6.center) to (1.center);
			\draw [bend left=90, looseness=1.25] (6.center) to (7.center);
			\draw (7.center) to (8.center);
			\draw (4.center) to (5.center);
			\draw [bend right=90, looseness=1.25] (4.center) to (3.center);
		\end{pgfonlayer}
	\end{tikzpicture} 
	~~~~~~~~~~~~ 
	(ii)~~~\gamma = \begin{tikzpicture}
		\begin{pgfonlayer}{nodelayer}
			\node [style=circle] (17) at (1, 0.25) {};
			\node [style=none] (18) at (1, 2.5) {};
			\node [style=none] (21) at (1.25, 2.25) {$b$};
		\end{pgfonlayer}
		\begin{pgfonlayer}{edgelayer}
			\draw (17) to (18.center);
		\end{pgfonlayer}
	\end{tikzpicture} := 	
	\begin{tikzpicture}
		\begin{pgfonlayer}{nodelayer}
			\node [style=circle] (0) at (-0.75, 1.5) {};
			\node [style=none] (1) at (-0.75, 0.5) {};
			\node [style=none] (2) at (-1.75, 0.5) {};
			\node [style=none] (3) at (-1.75, 2.25) {};
			\node [style=none] (4) at (-2, 2) {$b$};
			\node [style=none] (5) at (-1.25, -0.25) {$\epsilon_R$};
		\end{pgfonlayer}
		\begin{pgfonlayer}{edgelayer}
			\draw (0) to (1.center);
			\draw [bend left=90, looseness=1.50] (1.center) to (2.center);
			\draw (2.center) to (3.center);
		\end{pgfonlayer}
	\end{tikzpicture} 
\end{equation*}

A linear monoid $b  \linmonw a$ in an LDC can be equivalently defined to consist of a left linear monoid $(\eta_L, \epsilon_L): b \linmonwl a$, and a right linear monoid $(\eta_R, \epsilon_R): a \linmonwr b$ such that the monoid structures on $b$ coincide, i.e.\ the equations in \eqref{eq.linmon} hold.

The monoids in $\Set$ produce left linear monoids in $\Poly$ as follows.
\begin{lemma}
\label[lemma]{Lemma.Poly.leftLinMon}
     If $(M, *, u) : \Set$ is a monoid, then $M \yon \linmonwl \yon^M$ has the structure of a left linear monoid where the linear duality $(\coeval,\eval)\colon M\yon\dual\yon^M$ is given as in \cref{corr: linear-representable-duality-eta}.
 \end{lemma}
\begin{proof}
The functor $M\mapsto M\yon$ is strong monoidal by \cref{prop.lin_rep_strong_monoidal}, so it preserves monoids.
\end{proof}

Explicitly, the $\ox$-monoid structure on $M \yon$ induced by $M$ is as follows.
\begin{align}
	\mu_1 &:= (a_1, a_2) \mapsto a_1 * a_2 : M \times M \to M \label{eqn.1} \\ 
	\nu_1 &:= * \mapsto e : \{ * \} \to M \label{eqn.2}
\end{align}
The directional components $\mu^\sharp_M$ and $\nu^\sharp_M$ are trivial.

The following Lemma shows that in $\Poly$, every representable $\yon^A$ is a right linear monoid.
\begin{lemma}
\label[lemma]{Lemma.Poly.rightLinMon}
	In $\Poly$, for any set $A$, there is a right linear monoid $A \yon \linmonwr \yon^A$ where the linear duality $(\coeval,\eval)\colon A\yon\dual\yon^A$ is given as in \cref{corr: linear-representable-duality-eta}.
\end{lemma}
\begin{proof}
The functor $A\mapsto \yon^A$ is strong monoidal by \cref{prop.lin_rep_strong_monoidal}, so it preserves comonoids, and every set is a comonoid in a unique way, making $\yon^A$ a $\ox$-monoid.
\end{proof}

Explicitly, the $\ox$-monoid structure on $\yon^A$ induced by the unique comonoid structure on $A$ is as follows.
\begin{align}
	\mu^\sharp_! &:= a \mapsto (a,a) : A \to A \times A \label{eqn.3} \\ 
	\nu^\sharp_! &:= a \mapsto * : A \to \{ * \} \label{eqn.4}
\end{align}
The positional components $\mu_1$ and $\nu_1$ are trivial.

\subsection{Linear comonoids}

In this section, we simply run the same story except with $\otimes$-comonoids in place of $\otimes$-monoids, and implicitly also $\tri$-monoids in place of $\tri$-comonoids. If the reader does not need this detail, they can safely skip to \cref{Lemma.Poly.leftLinComon}.

\begin{definition}
[{\cite[Definition 8.22]{Sri21}}] In an LDC, a {\bf linear comonoid} $b \lincomonw  a$ consists of linear duals $(\eta_L, \epsilon_L): b \dual a$ and $(\eta_R, \epsilon_R): a \dual b$, and a $\ox$-comonoid $(b, \delta, \epsilon)$, such that the $\tri$-comonoid structures induced on $a$ by the two dualities coincide:
\begin{equation}
\label{eq.lincomon}
	{\it (i)}~~~ \begin{tikzpicture}
		\begin{pgfonlayer}{nodelayer}
			\node [style=circle] (0) at (1.7, 2.75) {};
			\node [style=none] (1) at (1.2, 2) {};
			\node [style=none] (2) at (2.2, 2) {};
			\node [style=none] (3) at (1.7, 3.5) {};
			\node [style=none] (4) at (1.95, 3.25) {$b$};
			\node [style=none] (5) at (0.7, 2) {};
			\node [style=none] (6) at (-0.3, 2) {};
			\node [style=none] (7) at (2.7, 3.5) {};
			\node [style=none] (8) at (2.7, 1.25) {};
			\node [style=none] (9) at (2.95, 1.75) {$a$};
			\node [style=oa] (10) at (0.2, 3) {};
			\node [style=none] (11) at (0.2, 3.75) {};
			\node [style=none] (12) at (-0.25, 3.65) {$a \tri a$};
			\node [style=none] (13) at (2.2, 4.15) {$\eta_R$};
		\end{pgfonlayer}
		\begin{pgfonlayer}{edgelayer}
			\draw [in=-165, out=90, looseness=1.25] (1.center) to (0);
			\draw [in=90, out=-15, looseness=1.25] (0) to (2.center);
			\draw (0) to (3.center);
			\draw [bend left=90, looseness=1.75] (1.center) to (5.center);
			\draw [bend right=90] (6.center) to (2.center);
			\draw [bend right=90, looseness=1.25] (7.center) to (3.center);
			\draw (7.center) to (8.center);
			\draw [in=90, out=-45] (10) to (5.center);
			\draw [in=90, out=-135, looseness=1.25] (10) to (6.center);
			\draw (10) to (11.center);
		\end{pgfonlayer}
	\end{tikzpicture} = \begin{tikzpicture}
			\begin{pgfonlayer}{nodelayer}
				\node [style=circle] (0) at (0.75, 2.75) {};
				\node [style=none] (1) at (1.25, 2) {};
				\node [style=none] (2) at (0.25, 2) {};
				\node [style=none] (3) at (0.75, 3.5) {};
				\node [style=none] (4) at (0.5, 3.25) {$b$};
				\node [style=none] (5) at (1.75, 2) {};
				\node [style=none] (6) at (2.75, 2) {};
				\node [style=none] (7) at (-0.25, 3.5) {};
				\node [style=none] (8) at (-0.25, 1.25) {};
				\node [style=none] (9) at (-0.5, 1.75) {$a$};
				\node [style=oa] (10) at (2.25, 3) {};
				\node [style=none] (11) at (2.25, 3.75) {};
				\node [style=none] (12) at (2.75, 3.65) {$a \tri a$};
				\node [style=none] (13) at (0.25, 4.15) {$\eta_L$};
			\end{pgfonlayer}
			\begin{pgfonlayer}{edgelayer}
				\draw [in=-15, out=90, looseness=1.25] (1.center) to (0);
				\draw [in=90, out=-165, looseness=1.25] (0) to (2.center);
				\draw (0) to (3.center);
				\draw [bend right=90, looseness=1.75] (1.center) to (5.center);
				\draw [bend left=90] (6.center) to (2.center);
				\draw [bend left=90, looseness=1.25] (7.center) to (3.center);
				\draw (7.center) to (8.center);
				\draw [in=90, out=-135] (10) to (5.center);
				\draw [in=90, out=-45, looseness=1.25] (10) to (6.center);
				\draw (10) to (11.center);
			\end{pgfonlayer}
		\end{tikzpicture}		
	    ~~~~~~~~
	    {\it (ii)} ~~~ 
	    \begin{tikzpicture}
	    	\begin{pgfonlayer}{nodelayer}
	    		\node [style=none] (0) at (1.5, 3.5) {};
	    		\node [style=none] (2) at (2.5, 3.5) {};
	    		\node [style=none] (3) at (2.5, 1) {};
	    		\node [style=none] (4) at (2.75, 1.75) {$a$};
	    		\node [style=circle, scale=1.5] (5) at (0.75, 1.25) {};
	    		\node [style=none] (6) at (0.75, 1.25) {$\bot$};
	    		\node [style=none] (7) at (0.75, 4.25) {};
	    		\node [style=circle] (8) at (0.75, 2) {};
	    		\node [style=none] (9) at (1.5, 2.75) {};
	    		\node [style=none] (13) at (2, 4.25) {$\eta_R$};
	    		\node [style=circle] (14) at (1.5, 2.75) {};
	    	\end{pgfonlayer}
	    	\begin{pgfonlayer}{edgelayer}
	    		\draw [bend right=90, looseness=1.50] (2.center) to (0.center);
	    		\draw (2.center) to (3.center);
	    		\draw (5) to (7.center);
	    		\draw [dotted, in=-90, out=0, looseness=1.25] (8) to (9.center);
	    		\draw (0.center) to (14);
	    	\end{pgfonlayer}
	    \end{tikzpicture} = \begin{tikzpicture}
	    	\begin{pgfonlayer}{nodelayer}
	    		\node [style=none] (0) at (2, 3.5) {};
	    		\node [style=none] (2) at (1, 3.5) {};
	    		\node [style=none] (3) at (1, 1) {};
	    		\node [style=none] (4) at (0.75, 1.75) {$a$};
	    		\node [style=circle, scale=1.5] (5) at (2.75, 1.25) {};
	    		\node [style=none] (6) at (2.75, 1.25) {$\bot$};
	    		\node [style=none] (7) at (2.75, 4.25) {};
	    		\node [style=circle] (8) at (2.75, 2) {};
	    		\node [style=none] (9) at (2, 2.75) {};
	    		\node [style=none] (13) at (1.5, 4.25) {$\eta_L$};
	    		\node [style=circle] (14) at (2, 2.75) {};
	    	\end{pgfonlayer}
	    	\begin{pgfonlayer}{edgelayer}
	    		\draw [bend left=90, looseness=1.50] (2.center) to (0.center);
	    		\draw (2.center) to (3.center);
	    		\draw (5) to (7.center);
	    		\draw [dotted, in=-90, out=180, looseness=1.25] (8) to (9.center);
	    		\draw (0.center) to (14);
	    	\end{pgfonlayer}
	    \end{tikzpicture}
\end{equation}
\end{definition}

Similar to the definition of linear monoids, we can fine-grain the definition of linear comonoids to non-cyclic duals.

\begin{definition}
	In an LDC, a {\bf left linear comonoid}, $(\eta_L, \epsilon_L): b \lincomonwl a$, consists of a linear dual $(\eta_L, \epsilon_L): b \dual a$ equipped with a $\ox$-comonoid structure $(b, \delta: b \ox b \to b, \gamma: b \to \yon)$ on the left dual.
\end{definition} 
The $\ox$-comonoid structure on $b$ induces a $\tri$-monoid structure $(a, \mu, \nu)$ as follows:
	\begin{equation*}
		{\it (i)} ~~~ \mu = \begin{tikzpicture}
		\begin{pgfonlayer}{nodelayer}
			\node [style=none] (19) at (4.5, 1.25) {};
			\node [style=none] (20) at (4.75, 1.5) {$a$};
			\node [style=oa] (26) at (4.4, 3.5) {};
			\node [style=none] (27) at (4.4, 4.25) {};
			\node [style=none] (28) at (3.75, 4.1) {$a \tri a$};
			\node [style=circle] (29) at (4.5, 2.25) {};
		\end{pgfonlayer}
		\begin{pgfonlayer}{edgelayer}
			\draw (26) to (27.center);
			\draw [bend left=60] (26) to (29);
			\draw [bend right=60] (26) to (29);
			\draw (29) to (19.center);
		\end{pgfonlayer}
	\end{tikzpicture} := \begin{tikzpicture}
			\begin{pgfonlayer}{nodelayer}
				\node [style=circle] (0) at (0.75, 2.75) {};
				\node [style=none] (1) at (1.25, 2) {};
				\node [style=none] (2) at (0.25, 2) {};
				\node [style=none] (3) at (0.75, 3.5) {};
				\node [style=none] (5) at (1.75, 2) {};
				\node [style=none] (6) at (2.75, 2) {};
				\node [style=none] (7) at (-0.25, 3.5) {};
				\node [style=none] (8) at (-0.25, 1.25) {};
				\node [style=none] (9) at (-0.5, 1.75) {$a$};
				\node [style=oa] (10) at (2.25, 3) {};
				\node [style=none] (11) at (2.25, 3.75) {};
				\node [style=none] (12) at (2.95, 3.65) {$a \tri a$};
				\node [style=none] (13) at (0.25, 4.15) {$\eta_L$};
			\end{pgfonlayer}
			\begin{pgfonlayer}{edgelayer}
				\draw [in=-15, out=90, looseness=1.25] (1.center) to (0);
				\draw [in=90, out=-165, looseness=1.25] (0) to (2.center);
				\draw (0) to (3.center);
				\draw [bend right=90, looseness=1.75] (1.center) to (5.center);
				\draw [bend left=90] (6.center) to (2.center);
				\draw [bend left=90, looseness=1.25] (7.center) to (3.center);
				\draw (7.center) to (8.center);
				\draw [in=90, out=-135] (10) to (5.center);
				\draw [in=90, out=-45, looseness=1.25] (10) to (6.center);
				\draw (10) to (11.center);
			\end{pgfonlayer}
		\end{tikzpicture}		
	    ~~~~~~~~
	    {\it (ii)} ~~~ \nu = \begin{tikzpicture}
	    	\begin{pgfonlayer}{nodelayer}
	    		\node [style=none] (19) at (4.5, 1.25) {};
	    		\node [style=none] (20) at (5, 1.75) {$a$};
	        		\node [style=none] (24) at (4.5, 4) {};
	    		\node [style=none] (26) at (5, 4) {$\bot$};
	    		\node [style=circle] (27) at (4.5, 3.5) {};
	    	\end{pgfonlayer}
	    	\begin{pgfonlayer}{edgelayer}
	    		\draw (19.center) to (27);
	    		\draw (27) to (24.center);
	    	\end{pgfonlayer}
	    \end{tikzpicture} :=  \begin{tikzpicture}
	    	\begin{pgfonlayer}{nodelayer}
	    		\node [style=none] (0) at (2, 3.5) {};
	    		\node [style=none] (2) at (1, 3.5) {};
	    		\node [style=none] (3) at (1, 1) {};
	    		\node [style=none] (4) at (0.75, 1.75) {$a$};
	    		\node [style=circle, scale=1.5] (5) at (2.75, 1.25) {};
	    		\node [style=none] (6) at (2.75, 1.25) {$\bot$};
	    		\node [style=none] (7) at (2.75, 4.25) {};
	    		\node [style=circle] (8) at (2.75, 2) {};
	    		\node [style=none] (9) at (2, 2.75) {};
	    		\node [style=none] (13) at (1.5, 4.25) {$\eta_L$};
	    		\node [style=circle] (14) at (2, 2.75) {};
	    	\end{pgfonlayer}
	    	\begin{pgfonlayer}{edgelayer}
	    		\draw [bend left=90, looseness=1.50] (2.center) to (0.center);
	    		\draw (2.center) to (3.center);
	    		\draw (5) to (7.center);
	    		\draw [dotted, in=-90, out=180, looseness=1.25] (8) to (9.center);
	    		\draw (0.center) to (14);
	    	\end{pgfonlayer}
	    \end{tikzpicture}
	    \end{equation*}

\begin{definition}
	In an LDC, a {\bf right linear comonoid} $(\eta_R, \epsilon_R): b \lincomonwr a$ consists of a linear dual $(\eta_R, \epsilon_R): b \dual a$ equipped with a $\ox$-comonoid structure $(a, \delta: a \to a \ox a, \gamma: a \to \yon)$ on the right dual. 
\end{definition}
The $\ox$-comonoid structure on $a$ induces a $\tri$-monoid structure $(b, \mu, \nu)$ as follows:
	\begin{equation*}
{\it (i)} ~~~ \mu = \begin{tikzpicture}
		\begin{pgfonlayer}{nodelayer}
			\node [style=none] (19) at (4.5, 1.25) {};
			\node [style=none] (20) at (4.75, 1.5) {$b$};
			\node [style=oa] (26) at (4.4, 3.5) {};
			\node [style=none] (27) at (4.4, 4.25) {};
			\node [style=none] (28) at (3.75, 4.1) {$b \tri b$};
			\node [style=circle] (29) at (4.5, 2.25) {};
		\end{pgfonlayer}
		\begin{pgfonlayer}{edgelayer}
			\draw (26) to (27.center);
			\draw [bend left=60] (26) to (29);
			\draw [bend right=60] (26) to (29);
			\draw (29) to (19.center);
		\end{pgfonlayer}
	\end{tikzpicture} :=\begin{tikzpicture}
		\begin{pgfonlayer}{nodelayer}
			\node [style=circle] (0) at (1.7, 2.75) {};
			\node [style=none] (1) at (1.2, 2) {};
			\node [style=none] (2) at (2.2, 2) {};
			\node [style=none] (3) at (1.7, 3.5) {};
			\node [style=none] (4) at (1.95, 3.25) {$a$};
			\node [style=none] (5) at (0.7, 2) {};
			\node [style=none] (6) at (-0.3, 2) {};
			\node [style=none] (7) at (2.7, 3.5) {};
			\node [style=none] (8) at (2.7, 1.25) {};
			\node [style=none] (9) at (2.95, 1.75) {$b$};
			\node [style=oa] (10) at (0.2, 3) {};
			\node [style=none] (11) at (0.2, 3.75) {};
			\node [style=none] (12) at (-0.5, 3.65) {$b \tri b$};
			\node [style=none] (13) at (2.2, 4.15) {$\eta_R$};
		\end{pgfonlayer}
		\begin{pgfonlayer}{edgelayer}
			\draw [in=-165, out=90, looseness=1.25] (1.center) to (0);
			\draw [in=90, out=-15, looseness=1.25] (0) to (2.center);
			\draw (0) to (3.center);
			\draw [bend left=90, looseness=1.75] (1.center) to (5.center);
			\draw [bend right=90] (6.center) to (2.center);
			\draw [bend right=90, looseness=1.25] (7.center) to (3.center);
			\draw (7.center) to (8.center);
			\draw [in=90, out=-45] (10) to (5.center);
			\draw [in=90, out=-135, looseness=1.25] (10) to (6.center);
			\draw (10) to (11.center);
		\end{pgfonlayer}
	\end{tikzpicture}     ~~~~~~~~
    {\it (ii)} ~~~ \nu = \begin{tikzpicture}
    	\begin{pgfonlayer}{nodelayer}
    		\node [style=none] (19) at (4.5, 1.25) {};
    		\node [style=none] (20) at (5, 1.75) {$b$};
        		\node [style=none] (24) at (4.5, 4) {};
    		\node [style=none] (26) at (5, 4) {$\bot$};
    		\node [style=circle] (27) at (4.5, 3.5) {};
    	\end{pgfonlayer}
    	\begin{pgfonlayer}{edgelayer}
    		\draw (19.center) to (27);
    		\draw (27) to (24.center);
    	\end{pgfonlayer}
    \end{tikzpicture} := 
    \begin{tikzpicture}
    	\begin{pgfonlayer}{nodelayer}
    		\node [style=none] (0) at (1.5, 3.5) {};
    		\node [style=none] (1) at (1.25, 3.5) {$a$};
    		\node [style=none] (2) at (2.5, 3.5) {};
    		\node [style=none] (3) at (2.5, 1) {};
    		\node [style=none] (4) at (2.75, 1.75) {$b$};
    		\node [style=circle, scale=1.5] (5) at (0.75, 1.25) {};
    		\node [style=none] (6) at (0.75, 1.25) {$\bot$};
    		\node [style=none] (7) at (0.75, 4.25) {};
    		\node [style=circle] (8) at (0.75, 2) {};
    		\node [style=none] (9) at (1.5, 2.75) {};
    		\node [style=none] (13) at (2, 4.25) {$\eta_R$};
    		\node [style=circle] (14) at (1.5, 2.75) {};
    	\end{pgfonlayer}
    	\begin{pgfonlayer}{edgelayer}
    		\draw [bend right=90, looseness=1.50] (2.center) to (0.center);
    		\draw (2.center) to (3.center);
    		\draw (5) to (7.center);
    		\draw [dotted, in=-90, out=0, looseness=1.25] (8) to (9.center);
    		\draw (0.center) to (14);
    	\end{pgfonlayer}
    \end{tikzpicture} 
	\end{equation*}

A {\bf linear comonoid} $b \lincomonw  a$  in any LDC can be equivalently defined to consist of a left linear comonoid $(\eta_L, \epsilon_L): b \lincomonwl a$ and a right linear comonoid $(\eta_R, \epsilon_R): a \lincomonwr b$ such that the $\ox$-comonoid structures on $b$ coincide and the equations in \eqref{eq.lincomon} hold.
		
The following Lemma shows that in $\Poly$, every linear polynomial $M \yon$ is a left linear comonoid: 
\begin{lemma}
\label[lemma]{Lemma.Poly.leftLinComon}
	In $\Poly$, for any set $A$, there is a left linear comonoid $A \yon \lincomonwl \yon^A$ where the linear duality $(\coeval,\eval)\colon A\yon\dual\yon^A$ is given as in \cref{corr: linear-representable-duality-eta}.
\end{lemma}
\begin{proof}
The functor $A\mapsto A\yon$ is strong monoidal by \cref{prop.lin_rep_strong_monoidal}, so it preserves comonoids, and every set is a comonoid in a unique way.
\end{proof}

Explicitly, the $\ox$-comonoid structure on $A \yon$ is as follows:
\begin{align}
	\delta_1 &:= a \mapsto (a,a) : A \to A \times A \label{eqn.5} \\ 
	\gamma_1 &:= a \mapsto * : A \to \{ * \} \label{eqn.6}
\end{align}

The monoids in $\Set$ produce right linear comonoids in $\Poly$ as follows.
\begin{lemma}
\label[lemma]{Lemma.Poly.rightLinComon}
     If $(M, *, u) : \Set$ is a monoid, then $M \yon \lincomonwl \yon^M$ has the structure of a right linear comonoid where the linear duality $(\coeval,\eval)\colon M\yon\dual\yon^M$ is given as in \cref{corr: linear-representable-duality-eta}.
 \end{lemma}
\begin{proof}
The functor $M\mapsto M\yon$ is strong monoidal by \cref{prop.lin_rep_strong_monoidal}, so it preserves monoids.
\end{proof}

\subsection{Linear bialgebras}

In LDCs, a linear bialgebra on $b\dual a$ arises from a linear monoid structure interacting bialgebraically with a linear comonoid structure. In this section, we again define a left and right versions of linear bialgebra in $\otimes$-symmetric LDCs, and then we proceed to identify them in $\Poly$. We will see that they can be identified with monoids in $\Set$. 

Let us first recall the definition of a bialgebra in a symmetric monoidal category (SMC).

\begin{definition}[\cite{Han08}]
In a SMC, a {\bf bialgebra} $\bialg{a}$ consists of a monoid $\monoid{a}$ and a comonoid $\comonoid{a}$ 
satisfying the following equations: 
\begin{equation}
\label{eqn.bialg}
\textit{(i)}~~~  \begin{tikzpicture}[yscale=-1]
	\begin{pgfonlayer}{nodelayer}
		\node [style=circle] (0) at (0, 4) {};
		\node [style=circle] (1) at (0, 5) {};
		\node [style=none] (2) at (-0.5, 3) {};
		\node [style=none] (3) at (0.5, 3) {};
	\end{pgfonlayer}
	\begin{pgfonlayer}{edgelayer}
		\draw [bend right, looseness=1.25] (0) to (2.center);
		\draw [bend left, looseness=1.25] (0) to (3.center);
		\draw (1) to (0);
	\end{pgfonlayer}
\end{tikzpicture} = \begin{tikzpicture}[yscale=-1]
	\begin{pgfonlayer}{nodelayer}
		\node [style=circle] (1) at (0, 5) {};
		\node [style=none] (3) at (0, 3) {};
		\node [style=circle] (4) at (0.75, 5) {};
		\node [style=none] (5) at (0.75, 3) {};
	\end{pgfonlayer}
	\begin{pgfonlayer}{edgelayer}
		\draw (3.center) to (1);
		\draw (5.center) to (4);
	\end{pgfonlayer}
\end{tikzpicture} ~~~~~~~~~~~~ \textit{(ii)} ~~~ \begin{tikzpicture}
	\begin{pgfonlayer}{nodelayer}
		\node [style=circle] (0) at (0, 4) {};
		\node [style=circle] (1) at (0, 5) {};
		\node [style=none] (2) at (-0.5, 3) {};
		\node [style=none] (3) at (0.5, 3) {};
	\end{pgfonlayer}
	\begin{pgfonlayer}{edgelayer}
		\draw [bend right, looseness=1.25] (0) to (2.center);
		\draw [bend left, looseness=1.25] (0) to (3.center);
		\draw (1) to (0);
	\end{pgfonlayer}
\end{tikzpicture} = \begin{tikzpicture}
	\begin{pgfonlayer}{nodelayer}
		\node [style=circle] (1) at (0, 5) {};
		\node [style=none] (3) at (0, 3) {};
		\node [style=circle] (4) at (0.75, 5) {};
		\node [style=none] (5) at (0.75, 3) {};
	\end{pgfonlayer}
	\begin{pgfonlayer}{edgelayer}
		\draw (3.center) to (1);
		\draw (5.center) to (4);
	\end{pgfonlayer}
\end{tikzpicture} ~~~~~~~~~~~ \textit{(iii)} ~~~ \begin{tikzpicture}
	\begin{pgfonlayer}{nodelayer}
		\node [style=circle] (1) at (0, 5) {};
		\node [style=circle] (2) at (0, 3) {};
	\end{pgfonlayer}
	\begin{pgfonlayer}{edgelayer}
		\draw (1) to (2);
	\end{pgfonlayer}
\end{tikzpicture} = \id_I  ~~~~~~~~~~~~ \textit{(iv)} ~~~ \begin{tikzpicture}
	\begin{pgfonlayer}{nodelayer}
		\node [style=circle] (0) at (0, 5.25) {};
		\node [style=circle] (1) at (1.25, 5.25) {};
		\node [style=circle] (2) at (0, 3.75) {};
		\node [style=circle] (3) at (1.25, 3.75) {};
		\node [style=none] (4) at (0, 3) {};
		\node [style=none] (5) at (1.25, 3) {};
		\node [style=none] (6) at (0, 6) {};
		\node [style=none] (7) at (1.25, 6) {};
		\node [style=none] (8) at (1.25, 3) {};
	\end{pgfonlayer}
	\begin{pgfonlayer}{edgelayer}
		\draw (1) to (2);
		\draw (0) to (3);
		\draw [bend right=45] (0) to (2);
		\draw [bend left=45] (1) to (3);
		\draw (3) to (8.center);
		\draw (2) to (4.center);
		\draw (6.center) to (0);
		\draw (7.center) to (1);
	\end{pgfonlayer}
\end{tikzpicture} = \begin{tikzpicture}
	\begin{pgfonlayer}{nodelayer}
		\node [style=none] (4) at (0, 3) {};
		\node [style=none] (6) at (0, 6) {};
		\node [style=none] (7) at (1.5, 6) {};
		\node [style=none] (8) at (1.5, 3) {};
		\node [style=circle] (9) at (0.75, 4) {};
		\node [style=circle] (10) at (0.75, 5.25) {};
	\end{pgfonlayer}
	\begin{pgfonlayer}{edgelayer}
		\draw [in=90, out=-165] (9) to (4.center);
		\draw [in=90, out=-15] (9) to (8.center);
		\draw (9) to (10);
		\draw [in=255, out=15] (10) to (7.center);
		\draw [in=285, out=165] (10) to (6.center);
	\end{pgfonlayer}
\end{tikzpicture} 
\end{equation}
\end{definition}

The final equation is often referred to as the {\em bialgebra rule}. 
\cref{eqn.bialg} $(i)$ -- $(iv)$ can equivalently be seen as the multiplication map acting as a comonoid morphism for the 
$\ox$-comonoid $A \ox A$ or the comultiplication acts as a monoid morphism for the 
$\ox$-monoid $A \ox A$ \cite[Lemma 6.19]{Sri21}. 

Linear bialgebras were introduced as structures suitable to study complementary quantum observables in isomix LDCs. 

\begin{definition}
[{\cite[Defn. 8.29]{Sri21}}]
	In a symmetric LDC, a {\bf linear bialgebra}, $\frac{(\eta,\epsilon)}{(\eta',\epsilon')}\!\!:\! b \linbialgwb a$, consists of:
	\begin{enumerate}[(i), nosep]
		\item a linear monoid, $(\eta,\epsilon)\!\!:\! b \linmonw a$, and
		\item a linear comonoid, $(\eta',\epsilon')\!\!:\! b \lincomonb a$,
	\end{enumerate}
	such that:
	\begin{enumerate}[(i), nosep]
		\item $\bialg{b}$ is a $\ox$-bialgebra, and 
		\item $\bialgb{a}$ is a $\tri$-bialgebra. 
\end{enumerate}
\end{definition}

As before we fine-grain the definition of linear bialgebra to non-cyclic duals. 
\begin{definition}
	\label{Defn: left linear bialg}
	A {\bf left linear bialgebra}, $\frac{(\eta_L,\epsilon_L)}{(\eta'_L,\epsilon'_L)}\!: b\! \linbialgwl a$ in a $\ox$-symmetric LDC consists of:
	\begin{enumerate}[(i), nosep]
		\item a left linear monoid, $(\eta_L,\epsilon_L)\!\!:\! b \linmonwl a$, denoted below by white $\circ$, and
		\item a left linear comonoid, $(\eta'_L,\epsilon'_L)\!\!:\! b \lincomonwl a$, denoted below by black $\bullet$,
	\end{enumerate}
	such that $\bialg{b}$ is a $\ox$-bialgebra, that is, the left dual is a bialgebra via the left linear monoid and comonoid.
\end{definition}

\begin{definition}
	\label{Defn: right linear bialg}
	A {\bf right linear bialgebra}, $\frac{(\eta_R,\epsilon_R)}{(\eta'_R,\epsilon'_R)}\!\!:\! b \linbialgwr a$ in a $\ox$-symmetric LDC consists of:
	\begin{enumerate}[(i), nosep]
		\item a right linear monoid, $(\eta_R,\epsilon_R)\!\!:\! b \linmonwr a$, denoted below by white $\circ$, and
		\item a right linear comonoid, $(\eta'_R,\epsilon'_R)\!\!:\! b \lincomonwr a$, denoted below by black $\bullet$,
	\end{enumerate}
	such that $\bialg{a}$  is a $\ox$-bialgebra, that is, the right dual is a bialgebra via the right linear monoid and comonoid.
\end{definition}

$\Poly$ is a $\ox$-symmetric LDC, and it has a left linear bialgebra and a right linear bialgebra for every monoid $M: \Set$.

\begin{lemma}
\label[lemma]{lemma.monoid_bialg}
Suppose $(M, *, e)$ is a monoid in $\Set$. 
\begin{enumerate}[i., nosep]
\item The left linear monoid $(\coeval, \eval): M \yon \linmonwl \yon^M$ as defined in \cref{Lemma.Poly.leftLinMon}, and left linear comonoid $(\coeval, \eval): M \yon \lincomonwl \yon^M$ as defined in \cref{Lemma.Poly.leftLinComon}  together produce a left linear bialgebra.
\item The right linear monoid $(\coeval, \eval): M \yon \linmonwr \yon^M$ as defined in \cref{Lemma.Poly.rightLinMon}, and right linear comonoid $(\coeval, \eval): M \yon \lincomonwr \yon^M$ as defined in \cref{Lemma.Poly.rightLinComon}  together produce a right linear bialgebra.
\end{enumerate}
\end{lemma}
\begin{proof}[Sketch of proof]
Using the $\ox$-monoid structure specified in \eqref{eqn.1} and \eqref{eqn.2} and the $\ox$-comonoid structure specified in \eqref{eqn.5} and \eqref{eqn.6} to prove that $M\yon$ is a $\ox$-bialgebra. The same idea applies to proving that the right linear monoid and comonoid gives a $\ox$-bialgebra on $\yon^M$.
\end{proof}

\begin{corollary}
In $\Poly$, we have that $M \yon \linbialgwl \yon^M$ is a left linear bialgebra and \\$M \yon \linbialgwr \yon^M$ is a right linear bialgebra iff $M$ is a monoid in $\Set$.
\end{corollary}
\begin{proof}
Direct consequence of \cref{prop.lin_rep_strong_monoidal,lemma.monoid_bialg,Thm: only Duals in Poly}.
\end{proof}



\refs

\bibitem [Aguiar and Mahajan (2010)]{AMa10}
Marcelo Aguiar and Swapneel Arvind Mahajan.
\newblock {Monoidal functors, species and Hopf algebras}, volume~29.
\newblock {\em CRM Monograph Series}, 2010. \newblock DOI: \url{https://doi.org/10.1090/crmm/029}

\bibitem [Blute et.al. (2000)]{BCS00}
Richard Blute, Robin Cockett, and Robert Seely.
\newblock Feedback for linearly distributive categories: traces and fixpoints.
\newblock {\em Journal of Pure and Applied Algebra}, 154(1-3):27--69, 2000.
\newblock DOI: \url{https://doi.org/10.1016/S0022-4049(99)00180-2}

\bibitem [Blute et.al. (1996)]{BCST96}
Richard Blute, Robin Cockett, Robert Seely, and Todd Trimble.
\newblock Natural deduction and coherence for weakly distributive categories.
\newblock {\em Journal of Pure and Applied Algebra}, 113(3):229--296, 1996.
\newblock DOI: \url{https://doi.org/10.1016/0022-4049(95)00159-X}

\bibitem [Cockett et. al. (2000)]{CKS00}
Robin Cockett, J{\"u}rgen Koslowski, and Robert Seely.
\newblock Introduction to linear bicategories.
\newblock {\em Mathematical Structures in Computer Science}, 10(2):165--203, 2000.
\newblock DOI: \url{https://doi.org/10.1017/S0960129520003047}

\bibitem [Cockett and Seely (1997a)]{CS97a} 
Robin Cockett and Robert Seely.
\newblock Proof theory for full intuitionistic linear logic, bilinear logic, and mix categories.
\newblock {\em Theory and Applications of categories}, 3(5):85--131, 1997.

\bibitem [Cockett and Seely (1997b)]{CS97}
Robin Cockett and Robert Seely.
\newblock Weakly distributive categories.
\newblock {\em Journal of Pure and Applied Algebra}, 114(2):133--173, 1997.
\newblock DOI: \url{https://doi.org/10.1016/0022-4049(95)00160-3}

\bibitem [Cockett and Seely (1999)]{CS99}
Robin Cockett and Robert Seely.
\newblock Linearly distributive functors.
\newblock {\em Journal of Pure and Applied Algebra}, 143(1-3):155--203, 1999.
\newblock DOI: \url{https://doi.org/10.1016/S0022-4049(98)00110-8}

\bibitem [Cockett and Pastro (2009)]{DP08}
Brian Day and Craig Pastro.
\newblock Note on Frobenius monoidal functors.
\newblock {\em New York J. Math}, 14(733-742):20, 2008.
\newblock DOI: \url{https://doi.org/10.48550/arXiv.0801.4107}

\bibitem [Fiore et.al. (2018)]{FGHW18}
Marcelo Fiore, Nicola Gambino, Martin Hyland, and Glynn Winskel.
\newblock Relative pseudomonads, Kleisli bicategories, and substitution monoidal structures.
\newblock {\em Selecta Mathematica}, 24(3):2791--2830, 2018.
\newblock DOI: \url{https://doi.org/10.1007/s00029-017-0361-3}

\bibitem [Gambino et. al. (2024)]{GGV24}
Nicola Gambino, Richard Garner, and Christina Vasilakopoulou.
\newblock Monoidal Kleisli bicategories and the arithmetic product of coloured symmetric sequences.
\newblock {\em Documenta Mathematica}, 29(3):627--702, 2024.
\newblock DOI: \url{https://doi.org/10.4171/dm/950}

\bibitem [Garner and Franco (2016)]{GaF16}
Richard Garner and Ignacio L{\'o}pez Franco.
\newblock Commutativity.
\newblock {\em Journal of Pure and Applied Algebra}, 220(5):1707--1751, 2016.
\newblock DOI: \url{https://doi.org/10.1016/j.jpaa.2015.09.003}

\bibitem [Maia and M{\'e}ndez (2008)]{MaM08}
Manuel Maia and Miguel M{\'e}ndez.
\newblock On the arithmetic product of combinatorial species.
\newblock {\em Discrete Mathematics}, 308(23):5407--5427, 2008.
\newblock DOI: \url{https://doi.org/10.48550/arXiv.math/0503436}

\bibitem [Niu and Spivak (2008)]{NiS23}
Nelson Niu and David I Spivak.
\newblock Polynomial functors: A mathematical theory of interaction.
\newblock {\em arXiv e-prints}, pages arXiv--2312, 2023.
\newblock DOI: \url{https://doi.org/10.48550/arXiv.2312.00990}

\bibitem [Porst (2008)]{Han08}
Hans-E. Porst.
\newblock On categories of monoids, comonoids, and bimonoids.
\newblock {\em Quaestiones Mathematicae}, 31(2):127--139, 2008.
\newblock DOI: \url{https://doi.org/10.2989/QM.2008.31.2.2.474}

\bibitem [Spivak (2022)]{Spi22}
David I Spivak.
\newblock A reference for categorical structures on Poly.
\newblock {\em arXiv preprint arXiv:2202.00534}, 2022.
\newblock DOI: \url{https://doi.org/10.48550/arXiv.2202.00534}

\bibitem [Srinivasan (2021)]{Sri21}
Priyaa Varshinee Srinivasan.
\newblock {\em Dagger linear logic and categorical quantum mechanics}.
\newblock PhD thesis, University of Calgary, Calgary, AB, 2021.
\newblock Available at \url{https://arXiv:2303.14231}.
\newblock DOI: \url{https://doi.org/10.48550/arXiv.2303.14231}

\endrefs

\appendix

\section*{Appendix A}

Unabridged version of the composite map in \eqref{eqn: phip}:

\begin{equation*}
\begin{tikzcd}[column sep=huge,row sep=large]
  a \ar[r, "u_\otimes^{-1}","\simeq"'] \ar[d, "\Phi_a"'] 
& \yon \otimes a \ar[r, "\coeval \otimes \id"]  
& [10pt] \left(\coclose{[a, \yon]}{\yon} \tri [a,\yon] \right) \ox a \ar[d, " \partial^R"] \\
\coclose{[a,\yon]}{\yon} 
& \coclose{[a,\yon]}{\yon} \tri \yon \ar[l, "u_\tri","\simeq"']  
& \coclose{[a, \yon]}{\yon} \tri (([a,\yon]) \ox a) \ar[l, "\eval"] \\ 
\end{tikzcd}
\end{equation*}

Unabridged version of the composite map in \eqref{eqn: psip}:

\begin{equation*}
\begin{tikzcd}[column sep=huge,row sep=large]
  \left[ \coclose{a}{\yon}, \yon \right] \otimes \yon \ar[r, "u_\otimes^{-1}","\simeq"'] \ar[d, "\Psi_a"']
& \left[ \coclose{a}{\yon}, \yon \right] \ar[r, "\id \otimes \coeval"] 
& \left[ \coclose{a}{\yon}, \yon \right] \ox \left( \coclose{a}{\yon} \tri a \right) \ar[d, "\partial^L"] \\ 
a 
& \yon \tri a \ar[l, "u_\tri", "\simeq"'] 
& \left( \left[ \coclose{a}{\yon}, \yon \right] \ox \coclose{a}{\yon} \right) 
 \tri a \ar[l, "\eval"] 
\end{tikzcd} 
\end{equation*}

\end{document}